\documentclass[12pt]{amsart}       
\usepackage{txfonts}
\usepackage{amssymb}
\usepackage{bbm}
\usepackage{eucal}
\usepackage{graphicx}
\usepackage{amsmath}
\usepackage{amscd}
\usepackage[all]{xy}           
\usepackage{amsfonts,latexsym}
\usepackage{xspace}
\usepackage{epsfig}
\usepackage{float}
\usepackage{mathrsfs} 
\usepackage{color}
\usepackage{fancybox}
\usepackage{colordvi}
\usepackage{multicol}
\usepackage{colordvi}
\usepackage{ifpdf}
\ifpdf
  \usepackage[colorlinks,final,backref=page,hyperindex]{hyperref}
\else
  \usepackage[colorlinks,final,backref=page,hyperindex,hypertex]{hyperref}
\fi
\usepackage[active]{srcltx} 
\usepackage{tikz}
\usepackage{graphicx}
\usepackage{CJK}
\newtheorem{theorem}{Theorem}[section]

\newtheorem{proposition}[theorem]{Proposition}

\newtheorem{corollary}[theorem]{Corollary}
\newtheorem{prop-def}{Proposition-Definition}[section]
\newtheorem{coro-def}{Corollary-Definition}[section]

\theoremstyle{definition}
\newtheorem{definition}[theorem]{Definition}
\newtheorem{remark}[theorem]{Remark}

\newcommand{\nc}{\newcommand}
\nc{\tred}[1]{\textcolor{red}{#1}}
\nc{\tblue}[1]{\textcolor{blue}{#1}}
\nc{\tgreen}[1]{\textcolor{green}{#1}}
\nc{\tpurple}[1]{\textcolor{purple}{#1}}
\nc{\btred}[1]{\textcolor{red}{\bf #1}}
\nc{\btblue}[1]{\textcolor{blue}{\bf #1}}
\nc{\btgreen}[1]{\textcolor{green}{\bf #1}}
\nc{\btpurple}[1]{\textcolor{purple}{\bf #1}}
\nc{\NN}{{\mathbb N}}

\nc{\NANZ}{\mathbb{Z}_{\ge 0}^{A\times \mathbb{N}}} 
\nc{\NANMO}{\mathbb{Z}_{\ge 0}^{I_A}} 

\nc{\ncsha}{{\mbox{\cyr X}^{\mathrm NC}}} \nc{\ncshao}{{\mbox{\cyr
X}^{\mathrm NC}_0}}

\newcommand{\delete}[1]{}

\delete{
\nc{\mlabel}[1]{\label{#1}}
\nc{\mcite}[1]{\cite{#1}}
\nc{\mref}[1]{\ref{#1}}
\nc{\meqref}[1]{\eqref{#1}}
\nc{\mbibitem}[1]{\bibitem{#1}}
}

\nc{\mlabel}[1]{\label{#1}{\hfill \hspace{1cm}{\bf{{\ }\hfill(#1)}}}}
\nc{\mcite}[1]{\cite{#1}{{\bf{{\ }(#1)}}}}
\nc{\mref}[1]{\ref{#1}{{\bf{{\ }(#1)}}}}
\nc{\meqref}[1]{\eqref{#1}{{\bf{{\ }(#1)}}}}
\nc{\mbibitem}[1]{\bibitem[\bf #1]{#1}}
\font\cyr=wncyr10 
\font\scyr=wncyr6
\nc{\sha}{{\mbox{\scyr X}}}
\nc{\shap}{{\mbox{\cyrs X}}} 
\nc{\shpr}{\diamond}    
\nc{\shp}{\ast} \nc{\shplus}{\shpr^+}
\nc{\shprc}{\shpr_c}    
\nc{\dep}{\mrm{dep}} \nc{\lc}{\lfloor} \nc{\rc}{\rfloor}
\nc{\db}{\leq_{\rm db}} \nc{\bfk}{\bf k}

\nc{\cala}{{\mathcal A}} \nc{\calb}{{\mathcal B}}
\nc{\calc}{{\mathcal C}}
\nc{\cald}{{\mathcal D}} \nc{\cale}{{\mathcal E}}
\nc{\calf}{{\mathcal F}} \nc{\calg}{{\mathcal G}}
\nc{\calh}{{\mathcal H}} \nc{\cali}{{\mathcal I}}
\nc{\call}{{\mathcal L}} \nc{\calm}{{\mathcal M}}
\nc{\caln}{{\mathcal N}} \nc{\calo}{{\mathcal O}}
\nc{\calp}{{\mathcal P}} \nc{\calr}{{\mathcal R}}
\nc{\cals}{{\mathcal S}} \nc{\calt}{{\mathcal T}}
\nc{\calu}{{\mathcal U}} \nc{\calw}{{\mathcal W}} \nc{\calk}{{\mathcal K}}
\nc{\calx}{{\mathcal X}} \nc{\CA}{\mathcal{A}}

\nc{\fraka}{{\mathfrak a}} \nc{\frakA}{{\mathfrak A}}
\nc{\frakb}{{\mathfrak b}} \nc{\frakB}{{\mathfrak B}}
\nc{\frakc}{{\mathfrak c}}
\nc{\frakD}{{\mathfrak D}} \nc{\frakF}{\mathfrak{F}}
\nc{\frakf}{{\mathfrak f}} \nc{\frakg}{{\mathfrak g}}
\nc{\frakH}{{\mathfrak H}} \nc{\frakL}{{\mathfrak L}}
\nc{\frakM}{{\mathfrak M}} \nc{\bfrakM}{\overline{\frakM}}
\nc{\frakm}{{\mathfrak m}} \nc{\frakP}{{\mathfrak P}}
\nc{\frakN}{{\mathfrak N}} \nc{\frakp}{{\mathfrak p}}
\nc{\frakS}{{\mathfrak S}} \nc{\frakT}{\mathfrak{T}}
\nc{\frakX}{{\mathfrak X}}

\font\cyr=wncyr10 \font\cyrs=wncyr7
\nc{\tao}[1]{\textcolor{red}{Tao:#1}}
\nc{\lir}[1]{\textcolor{red}{Li:#1}}
\nc{\zhi}[1]{\textcolor{blue}{Zhi: #1}}
\nc{\xing}[1]{\textcolor{purple}{Xing:#1}}
\nc{\revise}[1]{\textcolor{red}{#1}}

\newcommand{\wt}{\operatorname{wt}}
\newcommand{\Fib}{\mathcal{F}}

\newcommand{\uvec}{\mathbf{u}}

\newcommand{\jmathop}{\jmath}

\newcommand{\Aut}{\operatorname{Aut}}
\newcommand{\kk}{{\bfk}} \newcommand{\xx}{{\bf x}} 
\nc{\bl}{{\boldsymbol\ell}}

\providecommand{\bb}{\mathbf{b}}

\providecommand{\ext}{\mathbin{\odot}}

\providecommand{\mm}{\mathbf{m}}
\providecommand{\ee}{\mathbf{e}}

\begin{document}
\begin{CJK*}{GBK}{song}
\title[Fertility fibres and coproduct coefficients in the LOT Hopf algebra]{Fertility fibres and coproduct coefficients in the LOT Hopf algebra}
%
\author{Zhicheng Zhu}
\address{School of Mathematics and Statistics, Lanzhou University, Lanzhou, Gansu 730000, P.\,R. China}
\email{zhuzhch16@lzu.edu.cn}

\author{Jingtao Li}
\address{School of Mathematics and Statistics, Lanzhou University, Lanzhou, Gansu 730000, P.\,R. China}
\email{ljingtao2025@lzu.edu.cn}

\author{Xing Gao$^{*}$}\thanks{*Corresponding author}
\address{School of Mathematics and Statistics, Lanzhou University,
Lanzhou, 730000, China;
Gansu Provincial Research Center for Basic Disciplines of Mathematics
and Statistics, Lanzhou, 730070, China
}
\email{gaoxing@lzu.edu.cn}

\date{\today}
\begin{abstract}
We study fibres of the fertility map $\Phi$ from decorated rooted trees to decorated
multi-index monomials. For a multi-index $\mathbf{k}$ of weight $-1$, the fibre
$\mathcal F_{\mathbf{k}}=\{\,t:\Phi(t)=\xx^{\mathbf{k}}\,\}$ consists of all rooted
trees with decoration--fertility profile $\mathbf{k}$.  We consider its ordinary
cardinality $F_{\mathbf{k}}$, its symmetry-weighted cardinality $W_{\mathbf{k}}$, and
the coefficient mass $J_{\mathbf{k}}$ appearing in the tree expansion of the
transposed embedding $\jmath$.
We obtain an explicit formula and a functional equation for the weighted counts, and
an exact multiset recursion together with a cycle-index functional equation for the
ordinary counts.  We also introduce coefficient generating functions for the lowering
derivation $\bar\partial$, derive recursive and transport-array formulas for the
corresponding coefficients, and use them to refine the admissible-cut formula for the
coproduct in the LOT Hopf algebra.
\end{abstract}

\makeatletter
\@namedef{subjclassname@2020}{\textup{2020} Mathematics Subject Classification}
\makeatother
\subjclass[2020]{
16T30;   
05A15;   
05C05.  
}

\keywords{LOT Hopf algebra;  fertility map;
rooted trees; fibre enumeration;   cycle-index methods;
  coproduct coefficients.  }

\maketitle

\tableofcontents

\setcounter{section}{0}

\allowdisplaybreaks

\section{Introduction}\label{sec:introduction}

\subsection{Rooted trees and Hopf algebras}

Rooted trees and their Hopf algebras play a central role in numerical analysis,
rough paths, renormalization, and the algebraic study of nonlinear equations.  The
Butcher--Connes--Kreimer (BCK) Hopf algebra is the commutative graded
Hopf algebra generated by rooted forests, with coproduct described by admissible cuts.
It provides a universal combinatorial framework for rooted-tree expansions, including
the classical $B$-series of numerical analysis and the renormalization Hopf algebra of
Connes and Kreimer
\cite{Brouder2000,Butcher1963,Butcher1972,CalaqueEbrahimiFardManchon2011,ConnesKreimer1998,Foissy2002,GrossmanLarson1989}.
In the decorated setting, the basis elements are decorated rooted forests, and the
coproduct separates each forest into a pruning part and a trunk part by summing over
admissible cuts.

The BCK Hopf algebra is closely related to the Grossman--Larson Hopf algebra of rooted
trees: it may be viewed as a graded dual Hopf algebra after choosing a suitable
pairing.  On the algebraic side, the free pre-Lie algebra generated by a decoration
set $A$ is realized by $A$-decorated rooted trees, while the free Novikov algebra
admits a description in terms of decorated multi-index monomials
\cite{BalinskiiNovikov1985,ChapotonLivernet2001,DL2002,GelfandDorfman1979,GuinOudom2008,Manchon2011,Osborn1992}.
The bridge between these two descriptions is the fertility map
\[
\Phi(t)
=
\prod_{v\in V(t)}
\xx_{f(v)-1}^{d(v)},
\]
which sends a decorated rooted tree to the multi-index monomial recording the
decoration and fertility of each vertex.

\subsection{The LOT Hopf algebra of decorated multi-indices}

The Hopf algebra of decorated multi-indices was introduced by Linares, Otto and
Tempelmayr as a tree-free combinatorial framework in the study of rough paths and
regularity structures.  We refer to it as the LOT Hopf algebra.  Its basis is indexed
by decorated multi-indices rather than by rooted trees, and this makes it suitable for
situations in which tree expansions are compressed into multi-index data
\cite{Gubinelli2010,Hairer2014,LOT2023,Lyons1998,LyonsCaruanaLevy2007,LyonsVictoir2007}.
Recent works have further developed this multi-index viewpoint in regularity
structures, rough paths, renormalization, and numerical $B$-series
\cite{BrunedEbrahimiFardHou2025,BrunedHairerZambotti2019,BrunedHou2025,BrunedLinares2024,JacquesZambotti2023,Linares2023}.

One important algebraic description of the multi-index coproducts was given by Bruned
and Hou~\cite{BrunedHou2025}.  Starting from explicit descriptions of the corresponding
Grossman--Larson products, they derived explicit formulae for the relevant coproducts
by dualization.  Their approach provides an algebraic counterpart to cut-based
descriptions of coproducts.  A closely related but differently normalized
combinatorial description was obtained by Manchon et al.~\cite{ZGM2026}.  They
showed that the fertility map induces, by transposition, an embedding of the LOT Hopf
algebra into the BCK Hopf algebra, and they obtained an admissible-cut formula for the
coproduct of decorated multi-indices.

The two formalisms lead to equivalent coproducts after a change of normalization, but
they use different pairings and hence different notions of symmetry factor for
multi-indices.  Consequently, the coefficients appearing in the explicit coproduct
formulae are not the same.  In the present paper we follow the symmetry-factor
convention of~\cite{ZGM2026}, because this normalization is naturally adapted to the
admissible-cut interpretation and to the fibre enumeration studied below.

\subsection{Why enumerate fertility fibres?}
Under the transposed embedding in~\cite{ZGM2026}, a single multi-index monomial
$\xx^{\mathbf{k}}$ represents the whole fibre
\[
\mathcal F_{\mathbf{k}}
=
\{\,t:\Phi(t)=\xx^{\mathbf{k}}\,\},
\]
that is, the class of decorated rooted trees with decoration--fertility profile
$\mathbf{k}$.  Thus the fertility map $\Phi$ compresses tree-level information into
multi-index data, while the embedding $\jmath$ records how a compressed multi-index
object expands back into rooted trees.

The starting point of this paper is that this compression is useful but not
enumeratively transparent.  A multi-index monomial is not merely a formal label: it
represents a whole family of rooted trees.  It is therefore natural to ask how large
the fibre $\mathcal F_{\mathbf{k}}$ is, how the automorphisms of trees in this fibre
contribute to the embedding $\jmath$, and how the resulting coefficients interact
with the LOT coproduct.

This leads to two different enumeration problems.  The first is the
symmetry-weighted enumeration, which is naturally compatible with the Hopf embedding
and the symmetry-factor convention of~\cite{ZGM2026}.  The second is the ordinary
enumeration of isomorphism classes of rooted trees, where automorphisms are not
divided out.  In the ordinary case, the branches below a root form an unordered
collection, and this is why multiset and cycle-index methods enter the picture.

\subsection{Main results and significance}

The purpose of the present paper is to develop an enumerative theory of the fibres of
the fertility map.  For each multi-index $\mathbf{k}$ of weight $-1$, we study the
fibre
\[
\mathcal F_{\mathbf{k}}
=
\{\,t:\Phi(t)=\xx^{\mathbf{k}}\,\}
\]
through three quantities: the ordinary fibre count
$F_{\mathbf{k}}:=|\mathcal F_{\mathbf{k}}|$, the symmetry-weighted count
\[
W_{\mathbf{k}}
=
\sum_{t\in\mathcal F_{\mathbf{k}}}\frac{1}{\sigma(t)},
\]
and the coefficient mass
\[
J_{\mathbf{k}}
=
\sum_{t\in\mathcal F_{\mathbf{k}}}
\frac{\sigma(\xx^{\mathbf{k}})}{\sigma(t)}.
\]
The quantity $J_{\mathbf{k}}$ is the total coefficient appearing in the tree expansion
of $\jmath(\xx^{\mathbf{k}})$.

The main contributions of the paper are as follows.

\begin{enumerate}
    \item We relate the fibre $\mathcal F_{\mathbf{k}}$ to the transposed embedding
    $\jmath$ (Definition~\ref{def:fibre}) and to the three quantities $F_{\mathbf{k}}$, $W_{\mathbf{k}}$, and
    $J_{\mathbf{k}}$ (Definition~\ref{def:counting}).

    \item We prove an explicit formula for $W_{\mathbf{k}}$ by passing to labelled
    rooted trees and using the Pr\"ufer correspondence with prescribed fertilities (Theorem~\ref{thm:weighted-formula}).  This also gives a closed formula for the coefficient
    mass $W_{\mathbf{k}}$ and $L_{\mathbf{k}}$ (Proposition~\ref{prop:weighted-recursion}).

    \item We derive a root-decomposition recursion for the weighted fibre counts,
    which yields the functional equation for the weighted generating series
    $T(\mathbf u)$ (Theorem~\ref{thm:weighted-functional}).

    \item We obtain an exact recursion for the ordinary fibre cardinalities
    $F_{\mathbf{k}}$ (Theorem~\ref{thm:ordinary-recursion}).  Since ordinary enumeration does not divide by automorphism
    groups, the correct root decomposition is expressed in terms of multisets of
    branch isomorphism classes.

    \item We package the ordinary recursion into an Euler-product and cycle-index
    functional equation, using the classical cycle-index formalism of P\'olya theory (Proposition~\ref{prop:ordinary-functional}).

    \item We introduce coefficient generating functions for the lowering derivation
    $\bar\partial$ (Definition~\ref{def:lowering}), derive recursive formulas for the coefficients
    $C_{\mathbf{k},\boldsymbol\ell}$ and $D_{\mathbf{k},\boldsymbol\ell}$ (Propositions~\ref{prop:C-recursion} and~\ref{prop:D-recursion}), and give a
    transport-array formula for the corresponding transition coefficients (Theorem~\ref{thm:explicit-transition}).

    \item We use these coefficients to refine the admissible-cut formula for the LOT
    coproduct by expanding the full trunk terms in the basis of multi-index monomials (Corollary~\ref{coro:cut-coproduct-refined}).
\end{enumerate}


These results make precise how much tree-level information is compressed by the
fertility map $\Phi$.  They also separate two enumerative regimes that behave quite
differently: the symmetry-weighted regime attached to the embedding $\jmath$, and the
ordinary regime governed by unordered multisets of rooted branches.  Finally, the
coefficient generating functions for the lowering derivation provide a computable way
to expand the full trunk terms appearing in the admissible-cut coproduct.  In this
way, fibre enumeration, symmetry factors, and coproduct coefficients are connected in
a single framework.

This framework may be useful for future comparisons between tree-based and
multi-index-based expansions in rough paths, regularity structures, and related
Hopf-algebraic constructions.

\subsection{Organization of the paper}

Section~\ref{sec:fertility} recalls the fertility map, and introduces its fibres, and the coefficient
mass associated with the embedding $\jmath$.  Section~\ref{sec:weighted} proves the
weighted fibre formula for $W_{\mathbf k}$ (Theorem~\ref{thm:weighted-formula}) and the functional equation for $T(\mathbf u)$ (Theorem~\ref{thm:weighted-functional}).
Section~\ref{sec:ordinary} develops the ordinary fibre enumeration and its
cycle-index form (Proposition~\ref{prop:ordinary-functional}).  Section~\ref{sec:coeff} studies coefficient generating functions
for the lowering derivation (Propositions~\ref{prop:C-recursion} and~\ref{prop:D-recursion}) and applies them to the refinement of the admissible-cut coproduct (Corollary~\ref{coro:cut-coproduct-refined}).

\smallskip

\noindent\textbf{Notation.}
Throughout the paper, we work over a field $\mathbb K$ of characteristic zero, which serves as the base field
for all vector spaces, tensor products, algebras, coalgebras, and linear maps under
consideration.  The decoration set $A$ is assumed to be finite.
All multi-indices are indexed by $I_A:=A\times\mathbb Z_{\ge -1}$ and are finitely
supported.

\section{The fertility map, its fibres, and the basic coefficient series}\label{sec:fertility}

This section fixes the notation and proves the first structural facts needed for the fibre-enumeration program.
The central objects are the fertility map $\Phi$, its fibres, the associated counting functions, and the coefficient generating function attached to the transposed embedding.

We first fix the notation for decorated multi-indices and for the monomials that they determine.

\begin{definition}[\cite{BrunedLinares2024, LOT2023, ZGM2026}]\label{def:multi}
\begin{enumerate}
\item An {\bf $A$-decorated multi-index} is a finitely supported map
\[
\mathbf{k}:I_A\longrightarrow \mathbb Z_{\ge 0}, \quad (a,j)\mapsto \mathbf{k}(a,j).
\]
For simplicity, we write
$$k_j^a:=\mathbf{k}(a,j), \quad \kk=(k_j^a)_{a\in A,\; j\ge -1}.$$
The set of all $A$-decorated multi-indices is denoted by $\NANMO$.

\item
The {\bf monomial} associated with $\mathbf{k}$ is
\[
\xx^{\mathbf{k}}
:=
\prod_{(a,j)\in I_A} (x_j^a)^{k_j^a}
=
\prod_{a\in A}\prod_{j\ge -1}(x_j^a)^{k_j^a}.
\]

\item
The {\bf degree and weight} of $\mathbf{k}$ are defined by
\[
|\mathbf{k}|
=
\sum_{a\in A}\sum_{j\ge -1}k_j^a,
\qquad
\operatorname{wt}(\mathbf{k})
=
\sum_{a\in A}\sum_{j\ge -1}j\,k_j^a.
\]

\item Its {\bf symmetry factor} is
\begin{equation}\label{eq:sigma-monomial}
\sigma(\xx^{\mathbf{k}})
:=
\prod_{(a,j)\in I_A} k_j^a!
=
\prod_{a\in A}\prod_{j\ge -1} k_j^a!.
\end{equation}
\end{enumerate}
\end{definition}

The next definition introduces the rooted-tree side of the construction and the fertility map
from rooted trees to multi-index monomials.

\begin{definition}[\cite{GrossmanLarson1989,ZGM2026}]\label{def:fertility-map}
\begin{enumerate}
\item
An $A$-{\bf decorated rooted tree} is a finite rooted tree $t$ endowed with a {\bf decoration map}
$d:V(t)\to A$.

\item For a vertex $v\in V(t)$, the {\bf fertility} $f(v)$ is the number of incoming edges of $v$, equivalently, the number of children of $v$.
All edges are oriented towards the root.

\item For each pair $(a,j)\in I_A$, define
\begin{equation*}
k_j^a(t):=\#\{v\in V(t): d(v)=a,\ f(v)=j+1\}.
\end{equation*}
The {\bf profile} of $t$ is the multi-index $$\kk(t) :=(k_j^a(t))_{a\in A,\; j\ge -1},$$ 
and the {\bf fertility map} is
\begin{equation*}
\Phi(t):=\prod_{v\in V(t)} x_{f(v)-1}^{d(v)}=\xx^{\kk(t)}.
\end{equation*}
\end{enumerate}
\end{definition}

\begin{remark}\label{rk:balance}
Let $t$ be an $A$-decorated rooted tree, and write $\Phi(t)=\xx^{\kk}$. Then~\cite{ZGM2026}
\begin{equation*}
|\kk|=|V(t)|, \quad \sum_{a\in A}\sum_{j\ge -1}(j+1)k_j^a=|E(t)|=|V(t)|-1=|\kk|-1,
\end{equation*}
and therefore
$\wt(\kk)=-1.$
\end{remark}

Having defined the fertility map, we can group rooted trees according to their common
decoration--fertility profile.

\begin{definition}\label{def:fibre}
For any multi-index $\kk$, the {\bf fibre} of $\Phi$ over $\xx^{\kk}$ is defined to be
\begin{equation*}
\Fib_\kk:=\{\,t:\Phi(t)=\xx^{\kk}\,\},
\end{equation*}
where the set on the right is taken over all isomorphism classes of $A$-decorated rooted trees.
\end{definition}

Equivalently, $t\in \Fib_\kk$ if and only if $t$ has exactly $k_j^a$ vertices of decoration $a$ and fertility $j+1$ for every pair $(a,j)$.
Notice that the set $\Fib_\kk$ is finite for every multi-index $\kk$~\cite{ZGM2026}.

We list below several concepts that will be used frequently in the remainder of the paper.
To avoid confusing algebraic monomials with bookkeeping monomials in generating series,
we use $\xx^{\kk}$ for monomials in the LOT Hopf algebra and $\mathbf{u}^{\kk}$ for
monomials in the generating series, where
$\mathbf{u}=(u_{a,j})_{a\in A,\ j\ge -1}$ is a separate family of variables.

\begin{definition}\label{def:counting}
For each multi-index $\kk$, define
\begin{enumerate}
    \item the {\bf ordinary fibre cardinality}
$    F_\kk:=|\Fib_\kk|$;
    \item the {\bf symmetry-weighted fibre cardinality}
    \begin{equation*}
    W_\kk:=\sum_{t\in \Fib_\kk}\frac{1}{\sigma(t)}, \quad\text{where }\sigma(t):=|\Aut(t)|;
    \end{equation*}
    \item the {\bf coefficient mass}
    \begin{equation*}
    J_\kk:=\sum_{t\in \Fib_\kk}\frac{\sigma(\xx^{\kk})}{\sigma(t)}.
    \end{equation*}

\item For a family of variables $\uvec=(u_{a,j})_{a\in A,\; j\ge -1}$, set
\begin{equation*}
\uvec^\kk:=\prod_{a\in A}\prod_{j\ge -1} u_{a,j}^{k_j^a},
\end{equation*}
which is a finite product because $\kk$ has finite support.

\item The associated generating series are
\begin{equation*}
F(\uvec):=\sum_{\wt(\kk)=-1} F_\kk\,\uvec^\kk, \quad T(\uvec):=\sum_{\wt(\kk)=-1} W_\kk\,\uvec^\kk,
\end{equation*}
and the {\bf coefficient generating function}
\begin{equation*}
J(\uvec):=\sum_{\wt(\kk)=-1} J_\kk\,\uvec^\kk.
\end{equation*}
\end{enumerate}
\end{definition}

Let $\mathcal{V}_{\mathrm{LOT}}^A$ be the $\mathbb{K}$-vector space spanned by the monomials $\xx^{\kk}$ with $\wt(\kk)=-1$, and let $\mathcal{V}_{\mathrm{BCK}}^A$ be the $\mathbb{K}$-vector space spanned by isomorphism classes of finite $A$-decorated rooted trees.
The transposed embedding is the unique linear map~\cite{ZGM2026}
\begin{equation}\label{eq:jmath-def}
\jmathop:\mathcal{V}_{\mathrm{LOT}}^A\to \mathcal{V}_{\mathrm{BCK}}^A
\end{equation}
such that
\begin{equation*}
\langle t,\jmathop(\xx^{\kk})\rangle_{\mathrm{BCK}}
=
\langle \Phi(t),\xx^{\kk}\rangle_{\mathrm{LOT}}
\end{equation*}
for every decorated rooted tree $t$ and every monomial $\xx^{\kk}$ of weight $-1$. Here, the pairings on basis elements are given by   
\begin{equation*}
\langle \xx^{\kk},\xx^{\bl}\rangle_{\mathrm{LOT}}:=\sigma(\xx^{\kk})\,\delta_{\kk,{\bf \ell}}, \quad \langle t,u\rangle_{\mathrm{BCK}}:=\sigma(t)\,\delta_{t,u}.
\end{equation*}
 
For every multi-index $\kk$ with $\wt(\kk)=-1$, the map $\jmathop$ defined by~\eqref{eq:jmath-def} satisfies~\cite{ZGM2026}
\begin{equation}\label{eq:jmath-fibre}
\jmathop(\xx^{\kk})=\sum_{t\in \Fib_\kk}\frac{\sigma(\xx^{\kk})}{\sigma(t)}\,t.
\end{equation}

The following proposition explains why $J_{\kk}$ should be regarded as the total
coefficient mass of the fibre $\mathcal F_{\kk}$ in the tree expansion of $\jmath(\xx^\kk)$.

\begin{proposition}\label{prop:J-meaning}
For every multi-index $\kk$ with $\wt(\kk)=-1$, the total coefficient of the tree expansion~\eqref{eq:jmath-fibre} is exactly $J_\kk$.
Equivalently,
\begin{equation}\label{eq:J-total}
J_\kk=\sum_{t\in \Fib_\kk}[t]\jmathop(\xx^{\kk}),
\end{equation}
where $[t]\jmathop(\xx^{\kk})$ denotes the coefficient of $t$ in the basis expansion of $\jmathop(\xx^{\kk})$.
Consequently,
\begin{equation}\label{eq:J-series-meaning}
J(\uvec)=\sum_{\wt(\kk)=-1}\left(\sum_{t\in \Fib_\kk}[t]\jmathop(\xx^{\kk})\right)\uvec^\kk.
\end{equation}
\end{proposition}

\begin{proof}
By~\meqref{eq:jmath-fibre}, the coefficient of a tree $t\in \Fib_\kk$ in $\jmathop(\xx^{\kk})$ is $\sigma(\xx^{\kk})/\sigma(t)$, and the coefficient of any tree outside $\Fib_\kk$ is $0$.
Therefore
\begin{equation*}
\sum_{t\in \Fib_\kk}[t]\jmathop(\xx^{\kk}) = \sum_{t\in \Fib_\kk}\frac{\sigma(\xx^{\kk})}{\sigma(t)} = J_\kk
\end{equation*}
by Definition~\ref{def:counting}.
This proves~\eqref{eq:J-total}, and~\eqref{eq:J-series-meaning} is just the generating-series reformulation of the same identity.
\end{proof}

\section{Weighted fibre enumeration}\label{sec:weighted}
The goal of this section is twofold:
first, to prove a closed formula for the weighted fibre cardinality $W_\kk$; second, to derive a root-decomposition recursion and the corresponding functional equation for $T(\uvec)$.
To compute the symmetry-weighted fibre counts, we first pass to labelled rooted trees.

\begin{definition}
Let $I$ be a finite set.
A {\bf labelled $A$-decorated rooted tree} on $I$ is an $A$-decorated rooted tree whose vertex set is exactly $I$.
For any multi-index $\kk$ with $|\kk|=|I|$, define the {\bf labelled fibre} of $\Phi$ over $\xx^{\kk}$
on $I$ to be
\begin{equation*}
\Fib_\kk[I]:=\{\,T:\ V(T)=I,\ \Phi(T)=\xx^{\kk}\,\}.
\end{equation*}
\end{definition}

If $\wt(\kk)\neq -1$, then $\Fib_\kk[I]=\emptyset$ by Remark~\ref{rk:balance}.
If $|I|=|J|$, any bijection $I\to J$ transports labelled trees on $I$ bijectively to labelled trees on $J$,
preserving the profile.
Hence the cardinality of $\Fib_\kk[I]$ depends only on $|I|$.
For $|\kk|=n$, we therefore write
\begin{equation*}
L_\kk:=|\Fib_\kk[I]|
\quad\text{for any finite set } I \text{ with } |I|=n.
\end{equation*}

The following proposition explains why labelled fibres are the correct tool for studying
the symmetry-weighted counts.

\begin{proposition}\label{prop:labelled-to-weighted}
Let $\kk$ be a multi-index with $\wt(\kk)=-1$, and set $n:=|\kk|$.
Then
\begin{equation}\label{eq:Lk-Wk}
L_\kk=n!\,W_\kk.
\end{equation}
\end{proposition}

\begin{proof}
Let $I$ be a finite set with $|I|=n$.
For each isomorphism class $t\in \Fib_\kk$, choose a representative $T_t$.
Let
\begin{equation*}
\mathrm{Lab}(T_t,I):=\{\lambda:V(T_t)\to I \text{ bijective}\}.
\end{equation*}
Then there is a group action
\[
\Aut(T_t)\times \mathrm{Lab}(T_t,I) \longmapsto \mathrm{Lab}(T_t,I), \quad (\lambda,\alpha)\longmapsto \lambda\circ \alpha,
\]
which is free. Indeed, if $\lambda\circ \alpha=\lambda$, then $\alpha=\mathrm{id}$, as $\lambda$ is bijective.

Two labelings $\lambda,\mu\in \mathrm{Lab}(T_t,I)$ determine the same labelled rooted tree on $I$
if and only if they differ by an automorphism of $T_t$.
Therefore the number of distinct labelled trees on $I$ having underlying isomorphism class $t$ is exactly
\begin{equation*}
\frac{|\mathrm{Lab}(T_t,I)|}{|\Aut(T_t)|}
=
\frac{n!}{\sigma(t)}.
\end{equation*}
Summing over all isomorphism classes $t\in \Fib_\kk$, we obtain
\begin{align}
L_\kk =\sum_{t\in \Fib_\kk}\frac{n!}{\sigma(t)}  =n!\sum_{t\in \Fib_\kk}\frac{1}{\sigma(t)}
 =
n!\,W_\kk
\end{align}
by Definition~\ref{def:counting}. This proves~\eqref{eq:Lk-Wk}.
\end{proof}

We shall also use the following standard labelled-tree enumeration with prescribed
fertilities, which follows from the Pr\"ufer correspondence.

\begin{proposition}\label{prop:prescribed-fertilities}
Let $n\ge 1$, and let $r_1,\dots,r_n\in \mathbb{Z}_{\geq 0}$ satisfy
\begin{equation}\label{eq:fertility-sum}
r_1+\cdots+r_n=n-1.
\end{equation}
Then the number of rooted trees on the vertex set $[n]:=\{1,\dots,n\}$ such that
vertex $i$ has fertility $r_i$ for every $i$ is
\begin{equation}\label{eq:prescribed-fertility-count}
\frac{(n-1)!}{r_1!\cdots r_n!}.
\end{equation}
\end{proposition}

\begin{proof}
We first treat the case $n=1$.
Then~\eqref{eq:fertility-sum} forces $r_1=0$, and there is exactly one rooted tree on $[1]$.
Equation~\eqref{eq:prescribed-fertility-count} gives $0!/0!=1$, so the statement holds.

Assume now $n\ge 2$.
For any rooted tree on $[n]$, let $q$ denote its root.
Since $n\ge 2$, the root has at least one child, hence
$r_q\ge 1.$
Fix such a label $q$ with $r_q\ge 1$, and consider rooted trees on $[n]$ with root $q$
and fertility sequence $(r_1,\dots,r_n)$.
Forget the orientation and the root, and look at the underlying undirected tree.
Its degree sequence is then
\begin{equation}\label{eq:degree-sequence-root-q}
d_i=
\begin{cases}
r_i+1, & i\neq q,\\
r_q, & i=q.
\end{cases}
\end{equation}
Indeed, every non-root vertex has one edge to its parent and $r_i$ edges to its children,
whereas the root has no parent and only its $r_q$ child-edges.

Conversely, any undirected tree on $[n]$ with degree sequence~\eqref{eq:degree-sequence-root-q}
becomes uniquely a rooted tree with root $q$ and fertilities $(r_1,\dots,r_n)$
after orienting every edge towards $q$.
Thus the number of rooted trees with root $q$ and prescribed fertilities equals the number of
undirected labelled trees with degree sequence~\eqref{eq:degree-sequence-root-q}.

By the classical Pr\"ufer correspondence~\cite{Moon1970, Prufer1918}, labelled trees on $[n]$ are in bijection with sequences
of length $n-2$ over $[n]$, and vertex $i$ appears exactly $d_i-1$ times in the Pr\"ufer sequence.
Hence the number of undirected labelled trees with degree sequence $(d_1,\dots,d_n)$ is
\begin{equation}\label{eq:prufer-multinomial}
\frac{(n-2)!}{(d_1-1)!\cdots(d_n-1)!}.
\end{equation}
Applying~\eqref{eq:prufer-multinomial} to~\eqref{eq:degree-sequence-root-q}, we find that
the number of rooted trees with root $q$ and prescribed fertilities is
\begin{equation}\label{eq:count-fixed-root}
\frac{(n-2)!}{r_1!\cdots r_{q-1}!\,(r_q-1)!\,r_{q+1}!\cdots r_n!}.
\end{equation}
Now sum~\eqref{eq:count-fixed-root} over all possible roots $q$ satisfying $r_q\ge 1$.
The total number of rooted trees with fertilities $(r_1,\dots,r_n)$ is therefore
\begin{align*}
\sum_{q:\,r_q\ge 1}
\frac{(n-2)!}{r_1!\cdots r_{q-1}!\,(r_q-1)!\,r_{q+1}!\cdots r_n!}
&=
(n-2)!\sum_{q:\,r_q\ge 1}\frac{r_q}{r_1!\cdots r_n!}
\notag\\
&=
\frac{(n-2)!}{r_1!\cdots r_n!}\sum_{q=1}^n r_q
\notag\\
&=
\frac{(n-2)!}{r_1!\cdots r_n!}(n-1),
\end{align*}
where the last equality follows from~\eqref{eq:fertility-sum}. This is exactly~\eqref{eq:prescribed-fertility-count}.
\end{proof}

We now present the explicit formula for labelled and weighted fibres.
\begin{theorem}\label{thm:weighted-formula}
Let $\kk$ be a multi-index with $\wt(\kk)=-1$, and set $n:=|\kk|$.
Then
\begin{equation}\label{eq:Lk-explicit}
L_\kk=
\frac{n!(n-1)!}
{\displaystyle\prod_{a\in A}\prod_{j\ge -1} k_j^a!\,(j+1)!^{k_j^a}},
\end{equation}
and therefore
\begin{equation}\label{eq:Wk-explicit}
W_\kk=
\frac{(n-1)!}
{\displaystyle\prod_{a\in A}\prod_{j\ge -1} k_j^a!\,(j+1)!^{k_j^a}}.
\end{equation}
\end{theorem}

\begin{proof}
We count labelled trees on the label set $[n]$.
Let $\Theta_\kk$ be the set of maps
$$
\theta:[n]\to I_A
$$
such that, for every pair $(a,j)$,
exactly $k_j^a$ labels $i\in [n]$ satisfy $\theta(i)=(a,j)$.
Equivalently, $\theta$ assigns to each label a decoration--fertility type,
with multiplicities prescribed by $\kk$.
The number of such assignments is the multinomial coefficient
\begin{equation}\label{eq:Theta-k-cardinality}
|\Theta_\kk|
=
\frac{n!}
{\displaystyle\prod_{a\in A}\prod_{j\ge -1} k_j^a!}.
\end{equation}
Fix $\theta\in \Theta_\kk$ and write
\begin{equation}\label{eq:theta-components}
\theta(i)=(a_i,j_i),
\quad
r_i:=j_i+1
\quad
\text{for } i=1,\dots,n.
\end{equation}
Then $r_i\in \mathbb{Z}_{\geq 0}$ for every $i$.
Moreover,
\begin{align}
\sum_{i=1}^n r_i=\sum_{i=1}^n (j_i+1) =\sum_{a\in A}\sum_{j\ge -1}(j+1)k_j^a =|\kk| + \wt(\kk)=|\kk|-1
\label{eq:sum-ri}
\end{align}
Since $|\kk|=n$,~\eqref{eq:sum-ri} becomes
$
r_1+\cdots+r_n=n-1.
$

For the fixed type assignment $\theta$, a labelled decorated rooted tree on $[n]$
has profile $\kk$ and induces $\theta$ if and only if
the vertex $i$ has decoration $a_i$ and fertility $r_i$ for every $i$.
The decorations are already prescribed by $\theta$,
so the only remaining choice is the rooted tree structure with fertilities $(r_1,\dots,r_n)$.
By Proposition~\ref{prop:prescribed-fertilities}, the number of such rooted trees is
\begin{equation*}
\frac{(n-1)!}{r_1!\cdots r_n!}.
\end{equation*}
Using~\eqref{eq:theta-components}, we rewrite the denominator as
\begin{equation*}
r_1!\cdots r_n!
=
\prod_{a\in A}\prod_{j\ge -1}(j+1)!^{k_j^a}.
\end{equation*}
Hence the number of labelled decorated rooted trees compatible with the fixed $\theta$ is
\begin{equation}\label{eq:count-fixed-theta-2}
\frac{(n-1)!}
{\displaystyle\prod_{a\in A}\prod_{j\ge -1}(j+1)!^{k_j^a}}.
\end{equation}

Finally, every labelled decorated rooted tree of profile $\kk$ determines a unique $\theta\in\Theta_\kk$,
namely the map sending label $i$ to the pair $$(d(i), f(i)-1) ,$$
where $d(i)$ and $f(i)$ are the decoration and the fertility of the vertex $i$, respectively.
Therefore the families counted in~\eqref{eq:count-fixed-theta-2}, as $\theta$ ranges over $\Theta_\kk$,
form a partition of $\Fib_\kk[\,[n]\,]$.
Combining~\eqref{eq:Theta-k-cardinality} and~\eqref{eq:count-fixed-theta-2}, we obtain
\begin{align*}
L_\kk
&=
|\Theta_\kk|
\cdot
\frac{(n-1)!}
{\displaystyle\prod_{a\in A}\prod_{j\ge -1}(j+1)!^{k_j^a}}
\notag\\
&=
\frac{n!}
{\displaystyle\prod_{a\in A}\prod_{j\ge -1} k_j^a!}
\cdot
\frac{(n-1)!}
{\displaystyle\prod_{a\in A}\prod_{j\ge -1}(j+1)!^{k_j^a}}
\notag\\
&=
\frac{n!(n-1)!}
{\displaystyle\prod_{a\in A}\prod_{j\ge -1} k_j^a!\,(j+1)!^{k_j^a}},
\end{align*}
which proves~\eqref{eq:Lk-explicit}.
Dividing by $n!$ and using Proposition~\ref{prop:labelled-to-weighted} yields~\eqref{eq:Wk-explicit}.
\end{proof}

From~\eqref{eq:Wk-explicit}, the explicit formula for the coefficient mass follows directly.

\begin{corollary}
Let $\kk$ be a multi-index with $\wt(\kk)=-1$, and set $n:=|\kk|$.
Then
\begin{equation}\label{eq:Jk-explicit}
J_\kk=
\frac{(n-1)!}
{\displaystyle\prod_{a\in A}\prod_{j\ge -1}(j+1)!^{k_j^a}}.
\end{equation}
\end{corollary}

\begin{proof}
By Definition~\ref{def:counting},
\begin{equation*}
J_\kk=\sigma(\xx^{\kk})\,W_\kk.
\end{equation*}
Using~\eqref{eq:sigma-monomial} and~\eqref{eq:Wk-explicit},
\begin{equation*}
J_\kk = \left(\prod_{a\in A}\prod_{j\ge -1} k_j^a!\right)
\cdot \frac{(n-1)!} {\displaystyle\prod_{a\in A}\prod_{j\ge -1} k_j^a!\,(j+1)!^{k_j^a}} = \frac{(n-1)!}{\displaystyle\prod_{a\in A}\prod_{j\ge -1}(j+1)!^{k_j^a}},
\end{equation*}
which is~\eqref{eq:Jk-explicit}.
\end{proof}

As a consequence,  we immediately obtain explicit coefficient generating series.

\begin{corollary}
The generating series of Definition~\ref{def:counting} admit the coefficientwise expansions
\begin{equation*}
T(\uvec)
=
\sum_{\wt(\kk)=-1}
\frac{(|\kk|-1)!}
{\displaystyle\prod_{a\in A}\prod_{j\ge -1} k_j^a!\,(j+1)!^{k_j^a}}
\,\uvec^\kk,
\end{equation*}
and
\begin{equation*}
J(\uvec)
=
\sum_{\wt(\kk)=-1}
\frac{(|\kk|-1)!}
{\displaystyle\prod_{a\in A}\prod_{j\ge -1}(j+1)!^{k_j^a}}
\,\uvec^\kk.
\end{equation*}
\end{corollary}

\begin{proof}
It follows from inserting~\eqref{eq:Wk-explicit} and~\eqref{eq:Jk-explicit} into the definitions of
$T(\uvec)$ and $J(\uvec)$ in Definition~\ref{def:counting}.
\end{proof}

For every $(a,j)\in I_A$, let $\ee_j^a$ denote the multi-index
whose $(a,j)$-coordinate is $1$ and whose other coordinates are $0$. We now give a root-decomposition recursion for $W_\kk$.

\begin{proposition}\label{prop:weighted-recursion}
Let $\kk$ be a multi-index with $\wt(\kk)=-1$.
Then
\begin{equation}\label{eq:weighted-recursion}
W_\kk
=
\sum_{a\in A}\sum_{\substack{j\ge -1\\ k_j^a\ge 1}}
\frac{1}{(j+1)!}
\sum_{\substack{\kk^{1}+\cdots+\kk^{j+1}=\kk-\ee_j^a\\ \wt(\kk^{m})=-1,\ \forall m \in [j+1]}}
W_{\kk^{1}}\cdots W_{\kk^{j+1}}.
\end{equation}
\end{proposition}

\begin{proof}
Set $n:=|\kk|$, and fix a label set $I$ with $|I|=n$.
We count the labelled fibre $\Fib_\kk[I]$.

Take any $T\in \Fib_\kk[I]$, and let $r\in I$ be its root.
Suppose that the root has decoration $a$ and fertility $j+1$.
Then the root contributes exactly the profile $\ee_j^a$.
After removing the root, the remaining graph splits into exactly $j+1$ connected components,
namely the rooted subtrees hanging from the children of the root.
Let their profiles be $\kk^{1},\dots,\kk^{j+1}$.
Each $\kk^{m}$ has weight $-1$ by Remark~\ref{rk:balance},
and the profile additivity gives
\begin{equation}\label{eq:profile-root-split}
\kk=\ee_j^a+\kk^{1}+\cdots+\kk^{j+1}.
\end{equation}

Conversely, suppose that we are given:
\begin{enumerate}
    \item a root label $r\in I$; \label{it:rl}
    \item a pair $(a,j)$ with $k_j^a\ge 1$;
    \item an ordered $(j+1)$-tuple $(\kk^{1},\dots,\kk^{j+1})$ of multi-indices of weight $-1$
    satisfying~\eqref{eq:profile-root-split};
    \item an ordered partition
    \begin{equation*}
    I\setminus\{r\}=I_1\sqcup\cdots\sqcup I_{j+1}
    \end{equation*}
    such that $|I_m|=|\kk^{m}|$ for every $m$; \label{it:orp}
    \item for each $m$, a labelled decorated rooted tree $T_m\in \Fib_{\kk^{m}}[I_m]$. \label{it:ldsrt}
\end{enumerate}
Attach the roots of $T_1,\dots,T_{j+1}$ to a new root labelled $r$ and decorated by $a$.
The resulting rooted tree has profile $\kk$.
Hence the above data reconstruct a labelled tree in $\Fib_\kk[I]$.

Now fix $(a,j)$ and an ordered tuple $(\kk^{1},\dots,\kk^{j+1})$ as in
\eqref{eq:profile-root-split}, and set
\begin{equation*}
n_m:=|\kk^{m}|
\quad
\text{for } m=1,\dots,j+1.
\end{equation*}
The number of admissible choices in items (\ref{it:rl}), (\ref{it:orp}), and (\ref{it:ldsrt}) above is:
\begin{itemize}
    \item $n$ choices for the root label $r$;
    \item
    \begin{equation*}
    \frac{(n-1)!}{n_1!\cdots n_{j+1}!}
    \end{equation*}
    ordered partitions of $I\setminus\{r\}$ into blocks of sizes $n_1,\dots,n_{j+1}$;
    \item
    \begin{equation*}
    L_{\kk^{1}}\cdots L_{\kk^{j+1}}
    \end{equation*}
    choices of labelled trees on those blocks.
\end{itemize}
Therefore the number of resulting ordered constructions is
\begin{equation}\label{eq:ordered-construction-count}
n\cdot
\frac{(n-1)!}{n_1!\cdots n_{j+1}!}
\cdot
L_{\kk^{1}}\cdots L_{\kk^{j+1}}.
\end{equation}

However, the children of the root are intrinsically unordered.
A given labelled rooted tree in $\Fib_\kk[I]$ with root fertility $j+1$
arises from exactly $(j+1)!$ ordered listings of its $j+1$ child-subtrees.
Indeed, the $j+1$ connected components obtained after removing the root have pairwise disjoint vertex sets,
so every permutation of them gives a distinct ordered tuple in the construction above,
and no other ordered tuple yields the same rooted tree.
Hence we must divide~\eqref{eq:ordered-construction-count} by $(j+1)!$.

Summing over all admissible $(a,j)$ and all admissible ordered tuples
$(\kk^{1},\dots,\kk^{j+1})$, we obtain
\begin{align}
L_\kk
&=
\sum_{a\in A}\sum_{\substack{j\ge -1\\ k_j^a\ge 1}}
\frac{n}{(j+1)!}
\sum_{\substack{\kk^{1}+\cdots+\kk^{j+1}=\kk-\ee_j^a\\ \wt(\kk^{m})=-1,\ \forall m \in [j+1]}}
\frac{(n-1)!}{n_1!\cdots n_{j+1}!}
L_{\kk^{1}}\cdots L_{\kk^{j+1}}.
\label{eq:Lk-recursion}
\end{align}
Now use Proposition~\ref{prop:labelled-to-weighted}, namely
\begin{equation}\label{eq:Lk-to-Wk-each}
L_{\kk^{m}}=n_m!\,W_{\kk^{m}},
\quad
L_\kk=n!\,W_\kk.
\end{equation}
Substituting~\eqref{eq:Lk-to-Wk-each} into~\eqref{eq:Lk-recursion} and dividing by $n!$
gives exactly~\eqref{eq:weighted-recursion}.
\end{proof}

The root-decomposition recursion of Proposition~\ref{prop:weighted-recursion} can be
encoded equivalently as a functional equation for the weighted generating series $T(\mathbf u)$.

\begin{theorem}\label{thm:weighted-functional}
The formal series $T(\uvec)$ is the unique formal power series with zero constant term such that
\begin{equation}\label{eq:weighted-functional}
T(\uvec)
=
\sum_{a\in A}\sum_{j\ge -1}
u_{a,j}\,
\frac{T(\uvec)^{j+1}}{(j+1)!}.
\end{equation}
\end{theorem}

\begin{proof}
We endow each variable $u_{a,j}$ with weight $j$ and total degree $1$, so that the monomial $\uvec^\kk$ has weight $\wt(\kk)$
and total degree $|\kk|$. We first prove that the right-hand side of~\eqref{eq:weighted-functional} is a well-defined formal series.
Since $T(\uvec)$ has no constant term, $T(\uvec)^{j+1}$ has total degree at least $j+1$.
Therefore, for any fixed monomial $\uvec^\kk$ with total degree $|\kk|$, only finitely many pairs $(a,j)$ and finitely many decompositions
of $\kk-\ee_j^a$ can contribute to its coefficient.

Let us now compute the coefficient of $\uvec^\kk$ in the right-hand side of~\eqref{eq:weighted-functional}.
If $\wt(\kk)\neq -1$, then this coefficient is $0$, because every monomial occurring in
$u_{a,j}T(\uvec)^{j+1}$ has weight
\begin{equation*}
j+(j+1)(-1)=-1.
\end{equation*}
Assume therefore that $\wt(\kk)=-1$.
By the Cauchy product formula, the coefficient of $\uvec^\kk$ in the right-hand side is
\begin{equation*}
\sum_{a\in A}\sum_{\substack{j\ge -1\\ k_j^a\ge 1}}
\frac{1}{(j+1)!}
\sum_{\substack{\kk^{1}+\cdots+\kk^{j+1}=\kk-\ee_j^a\\ \wt(\kk^{m})=-1,\ \forall m \in [j+1]}}
W_{\kk^{1}}\cdots W_{\kk^{j+1}}.
\end{equation*}
By Proposition~\ref{prop:weighted-recursion}, this is exactly $W_\kk$,
which is the coefficient of $\uvec^\kk$ in $T(\uvec)$.
Hence~\eqref{eq:weighted-functional} holds coefficientwise.

It remains to prove uniqueness.
Let $S(\uvec)$ be another formal series with zero constant term satisfying~\eqref{eq:weighted-functional}.
We show by induction on the total degree that the coefficient of every monomial in $S(\uvec)$
is uniquely determined.
For total degree $1$, the only contributing terms on the right-hand side come from $j=-1$,
because $S(\uvec)^{j+1}$ has degree at least $j+1$.
Hence the degree-$1$ part is forced to be
\begin{equation*}
\sum_{a\in A} u_{a,-1}.
\end{equation*}

Assume now that all coefficients of total degree $<n$ are uniquely determined.
To obtain a monomial of total degree $n$ from the right-hand side of~\eqref{eq:weighted-functional},
the term with index $j=-1$ contributes only to degree $1$ and is therefore irrelevant when $n\ge 2$.
For $j \ge 0$, every contributing monomial in $S(\uvec)^{j+1}$ that can contribute to the degree-$n$
coefficient of the right-hand side is a product of $j+1$ monomials of positive degree whose
total degree is $n-1$; hence each factor has degree at most $n-1$, and therefore strictly
smaller than $n$.
Therefore the degree-$n$ coefficient of the right-hand side depends only on coefficients of $S(\uvec)$
of degrees strictly smaller than $n$, which are already fixed by the induction hypothesis.
Hence the degree-$n$ coefficient is uniquely determined.
This proves uniqueness of the solution with zero constant term.
\end{proof}

\section{Ordinary fibre enumeration via cycle-index methods}\label{sec:ordinary}
The weighted theory of Section~\ref{sec:weighted} led to the explicit formula~\eqref{eq:Wk-explicit}.
For the ordinary counts $F_\kk$, automorphisms are no longer divided out termwise, and the correct
combinatorics is that of multisets. The aim of this section is to derive an exact recursion for $F_\kk$
and to package it into an ordinary generating-series equation of cycle-index type.

To describe the ordinary fibre count, we first introduce the data that record how many
branches of each profile occur below the root.

\begin{definition}\label{def:profile-data}
Let
\begin{equation*}
\mathcal K_{-1}:=\{\,\bl:\ \bl \text{ is a multi-index and } \wt(\bl)=-1\,\}.
\end{equation*}
A {\bf profile-multiplicity datum} is a finitely supported family
\begin{equation*}
\nu=(\nu_\bl)_{\bl\in\mathcal K_{-1}},
\quad\text{where each }\nu_\bl\in \mathbb{Z}_{\geq 0}.
\end{equation*}
For such a family, define
\begin{equation*}
|\nu|:=\sum_{\bl\in\mathcal K_{-1}} \nu_\bl, \quad \Sigma(\nu):=\sum_{\bl\in\mathcal K_{-1}} \nu_\bl\,\bl.
\end{equation*}
\end{definition}

For $r,m \in \mathbb{Z}_{\geq 0}$, define the {\bf multiset number}
\begin{equation*}
\mathrm{Mlt}(r,m):=[z^m](1-z)^{-r},
\end{equation*}
where $[z^m](1-z)^{-r}$ denotes the coefficient of $z^m$ in $(1-z)^{-r}$.
Then
\begin{equation*}
\mathrm{Mlt}(r,m)=
\begin{cases}
\displaystyle \binom{r+m-1}{m}, & r\ge 1,\\[2ex]
\delta_{0,m}, & r=0.
\end{cases}
\end{equation*}
Moreover, if $X$ is a finite set with $|X|=r$, then $\mathrm{Mlt}(r,m)$ is exactly the number of multisets
of cardinality $m$ with elements in $X$.
 
With this notation, the ordinary fibre count is obtained by choosing the root type and
then a multiset of branch isomorphism classes.

\begin{theorem}\label{thm:ordinary-recursion}
Let $\kk$ be a multi-index with $\wt(\kk)=-1$.
Then
\begin{equation}\label{eq:ordinary-recursion}
F_{\kk}
=
\sum_{a\in A}
\sum_{\substack{j\ge -1\\ k_j^a\ge 1}}
\sum_{\substack{
\nu=(\nu_{\bl})_{\bl\in\mathcal K_{-1}}\\
\nu \text{ finitely supported}\\
|\nu|=j+1,\ \Sigma(\nu)=\kk-\ee_j^a
}}
\prod_{\bl\in\mathcal K_{-1}}
\operatorname{Mlt}(F_{\bl},\nu_{\bl}).
\end{equation}
The product is finite because $\nu$ is finitely supported.
\end{theorem}

\begin{proof}
We shall classify the trees in $\Fib_\kk$ by their root type and by the multiset of isomorphism classes of
their branches.
Let $t\in \Fib_\kk$.
Denote by $\rho$ the root of $t$.
Let $a=d(\rho)$ be the decoration of the root, and let $j+1=f(\rho)$ be its fertility.
Since the root contributes one vertex of decoration $a$ and fertility $j+1$, we have
$k_j^a\ge 1.$

Remove the root $\rho$ and all edges attached to it.
The remaining graph splits into exactly $j+1$ connected components, each of which is naturally a rooted
$A$-decorated tree, rooted at the child that used to be attached to $\rho$.
Let these rooted trees be
$
t_1,\dots,t_{j+1}.
$
For each $m\in\{1,\dots,j+1\}$, write
\begin{equation*}
\Phi(t_m)=\xx^{\bl^{m}}
\quad\text{with }
\bl^{m}\in \mathcal K_{-1},
\end{equation*}
which follows from Remark~\ref{rk:balance}.
Since the root contributes exactly $\ee_j^a$, and the profiles of the branches add, we obtain
\begin{equation}\label{eq:root-branch-additivity}
\kk=\ee_j^a+\bl^{1}+\cdots+\bl^{j+1}.
\end{equation}

Now group the branches by profile.
For each $\bl\in\mathcal K_{-1}$, define
\begin{equation*}
\nu_\bl:=\#\{\,m\in\{1,\dots,j+1\}: \bl^{m}=\bl\,\}.
\end{equation*}
Then $\nu=(\nu_\bl)_{\bl\in\mathcal K_{-1}}$ is finitely supported, and
\begin{equation}\label{eq:nu-from-tree-size}
|\nu|=\sum_{\bl\in\mathcal K_{-1}}\nu_\bl=j+1,
\end{equation}
because there are exactly $j+1$ branches, while
\begin{equation*}
\Sigma(\nu)=\sum_{\bl\in\mathcal K_{-1}}\nu_\bl\,\bl
=
\bl^{1}+\cdots+\bl^{j+1}
=
\kk-\ee_j^a
\end{equation*}
by~\eqref{eq:root-branch-additivity}.
Thus every tree $t\in\Fib_\kk$ determines a triple
$
(a,j,\nu)
$
satisfying the constraints appearing in~\eqref{eq:ordinary-recursion}.

We now count how many trees yield a fixed admissible triple $(a,j,\nu)$.
Fix such a triple.
For each profile $\bl\in\mathcal K_{-1}$, the fibre $\Fib_\bl$ is a finite set of cardinality $F_\bl$.
To build a tree in $\Fib_\kk$ with root decoration $a$, root fertility $j+1$, and branch-profile multiplicity
datum $\nu$, it is necessary and sufficient to choose, for every $\bl\in\mathcal K_{-1}$, a multiset
\begin{equation}\label{eq:multiset-choice}
M_\bl
\end{equation}
of cardinality $\nu_\bl$ with elements in the finite set $\Fib_\bl$.
Indeed, after all these multisets are chosen, one forms their disjoint union
\begin{equation*}
M:=\bigsqcup_{\bl\in\mathcal K_{-1}} M_\bl,
\end{equation*}
which is a multiset of exactly $j+1$ rooted trees, because~\eqref{eq:nu-from-tree-size} gives
\begin{equation*}
|M|=\sum_{\bl\in\mathcal K_{-1}}\nu_\bl=j+1.
\end{equation*}
Since each $F_\bl$ consists of isomorphism classes, choose one representative tree from
each selected class in the multiset $M$. Attaching their roots to a new root decorated
by $a$ yields an isomorphism class of decorated rooted trees, independent of the
chosen representatives.

We claim that this construction establishes a bijection between:
\begin{enumerate}
    \item trees in $\Fib_\kk$ with associated triple $(a,j,\nu)$;
    \item families $(M_\bl)_{\bl\in\mathcal K_{-1}}$ of multisets as in~\eqref{eq:multiset-choice}.
\end{enumerate}

The surjectivity is immediate from the construction.
For injectivity, suppose two such families produce rooted trees $t$ and $t'$.
Any isomorphism between $t$ and $t'$ must send the root to the root, because the root is the unique
vertex with no parent.
After removing the roots, such an isomorphism induces a bijection between the branch components of $t$
and those of $t'$, preserving isomorphism classes.
Hence the multisets of branch isomorphism classes must coincide.
Conversely, if the branch multisets coincide, then one can match equal branch isomorphism classes and
glue the corresponding branch isomorphisms to an isomorphism of the full rooted trees.
Therefore the resulting rooted tree is determined, up to isomorphism, exactly by the family of multisets
$(M_\bl)_\bl$.

For each fixed $\bl$, the number of possible multisets
$M_\bl$ is $\mathrm{Mlt}(F_\bl,\nu_\bl)$. The choices are independent for different $\bl$, so the number of families $(M_\bl)_\bl$ is
\begin{equation}\label{eq:product-mlt}
\prod_{\bl\in\mathcal K_{-1}} \mathrm{Mlt}(F_\bl,\nu_\bl).
\end{equation}
Summing~\eqref{eq:product-mlt} over all admissible triples $(a,j,\nu)$ proves~\eqref{eq:ordinary-recursion}.
\end{proof}

We now package the multiset choices appearing in Theorem~\ref{thm:ordinary-recursion}
into auxiliary generating series.

\begin{definition}\label{def:Hm}
For every integer $r\ge 1$, and for a family of variables $\uvec= (u_{a,j})_{a\in A,\; j\ge -1}$, define
\begin{equation*}
\uvec^{[r]}:=(u_{a,j}^r)_{a\in A,\; j\ge -1}.
\end{equation*}
For every integer $m\ge 0$, define the formal power series
\begin{equation*}
H_m(\uvec)
:=
\sum_{\substack{\nu \text{ finitely supported on }\mathcal K_{-1}\\ |\nu|=m}}
\left(
\prod_{\bl\in\mathcal K_{-1}}
\mathrm{Mlt}(F_\bl,\nu_\bl)
\right)
\uvec^{\Sigma(\nu)}.
\end{equation*}
\end{definition}

The auxiliary series $H_m(\mathbf u)$ admit a product expression, or equivalently a
plethystic exponential form.

\begin{proposition}\label{prop:ordinary-product}
The generating series of the formal power series $H_m(\uvec)$ satisfies
\begin{equation}\label{eq:Hm-product}
\sum_{m\ge 0} H_m(\uvec)\,z^m
=
\prod_{\bl\in\mathcal K_{-1}} (1-z\uvec^\bl)^{-F_\bl}
=
\exp\left(
\sum_{r\ge 1}\frac{z^r}{r}\,F(\uvec^{[r]})
\right).
\end{equation}
All identities are understood coefficientwise in the ring of formal power series in $z$ and the variables
$u_{a,j}$.
\end{proposition}

\begin{proof}
We first prove the equality between the left-hand side and the infinite product.
For each $\bl\in\mathcal K_{-1}$, Definition~\ref{def:profile-data} gives
\begin{equation}\label{eq:single-factor-expansion}
(1-z\uvec^\bl)^{-F_\bl}
=
\sum_{n\ge 0}\mathrm{Mlt}(F_\bl,n)\,z^n\uvec^{n\bl}.
\end{equation}
Multiplying~\eqref{eq:single-factor-expansion} over all $\bl\in\mathcal K_{-1}$ suggests the expansion
\begin{equation}\label{eq:formal-product-expansion}
\prod_{\bl\in\mathcal K_{-1}} (1-z\uvec^\bl)^{-F_\bl}
=
\sum_{\nu}
\left(
\prod_{\bl\in\mathcal K_{-1}}\mathrm{Mlt}(F_\bl,\nu_\bl)
\right)
z^{|\nu|}\uvec^{\Sigma(\nu)},
\end{equation}
where the sum runs over all finitely supported families $\nu$.
To justify this coefficientwise, fix a monomial $z^m\uvec^\kk$.
If a family $\nu$ contributes to the coefficient of $z^m\uvec^\kk$, then
$
|\nu|=m$ and
$\Sigma(\nu)=\kk.
$
In particular, if $\nu_\bl>0$, then $\bl$ is componentwise bounded by $\kk$, because
$\Sigma(\nu)=\kk$ is a sum of nonnegative multiples of multi-indices.
Since $\kk$ has finite support, there are only finitely many multi-indices $\bl$ componentwise bounded by $\kk$.
Hence only finitely many $\bl$ can occur in a contributing family $\nu$, and so the coefficient of
$z^m\uvec^\kk$ is a finite sum.
Therefore~\eqref{eq:formal-product-expansion} is valid coefficientwise.

Now group in~\eqref{eq:formal-product-expansion} the terms of fixed $z$-degree $m$.
Using Definition~\ref{def:Hm}, we obtain
\begin{equation}\label{eq:group-by-m}
\prod_{\bl\in\mathcal K_{-1}} (1-z\uvec^\bl)^{-F_\bl}
=
\sum_{m\ge 0} H_m(\uvec)\,z^m,
\end{equation}
which proves the first equality in~\eqref{eq:Hm-product}.

We next prove the exponential form.
Since the product in~\eqref{eq:group-by-m} has constant term $1$, its formal logarithm is well-defined.
Using the power-series identity $-\log(1-w)=\sum_{r\ge 1} w^r/r$, we get
\begin{align}
\log\!\left(\prod_{\bl\in\mathcal K_{-1}} (1-z\uvec^\bl)^{-F_\bl}\right)
&=
-\sum_{\bl\in\mathcal K_{-1}} F_\bl\,\log(1-z\uvec^\bl)
\notag\\
&=
\sum_{\bl\in\mathcal K_{-1}} F_\bl \sum_{r\ge 1}\frac{(z\uvec^\bl)^r}{r}
\notag\\
&=
\sum_{r\ge 1}\frac{z^r}{r}
\sum_{\bl\in\mathcal K_{-1}} F_\bl\,\uvec^{r\bl}.
\label{eq:log-product}
\end{align}
Here again the manipulations are coefficientwise legitimate:
for a fixed coefficient of $z^m\uvec^\kk$, only finitely many pairs $(r,\bl)$ can contribute, because
$r\le m$ and $\bl$ is componentwise bounded by $\kk$.
By Definition~\ref{def:Hm},
\begin{equation*}
(\uvec^{[r]})^\bl=\uvec^{r\bl}.
\end{equation*}
Therefore
\begin{equation}\label{eq:F-substitution}
\sum_{\bl\in\mathcal K_{-1}} F_\bl\,\uvec^{r\bl}
=
\sum_{\bl\in\mathcal K_{-1}} F_\bl\,(\uvec^{[r]})^\bl
=
F(\uvec^{[r]}),
\end{equation}
because $F(\uvec)=\sum_{\wt(\bl)=-1} F_\bl \uvec^\bl$ by Definition~\ref{def:counting}.
Substituting~\eqref{eq:F-substitution} into~\eqref{eq:log-product} yields
\begin{equation*}
\log\!\left(\prod_{\bl\in\mathcal K_{-1}} (1-z\uvec^\bl)^{-F_\bl}\right)
=
\sum_{r\ge 1}\frac{z^r}{r}\,F(\uvec^{[r]}).
\end{equation*}
Exponentiating both sides proves the second equality in~\eqref{eq:Hm-product}.
\end{proof}

The species $\mathrm{SET}_m$ encodes unordered sets of cardinality $m$~\cite{BLL}.
In more details,
\[
\mathrm{SET}_m[U]
:=
\begin{cases}
\{U\}, & \text{ if } |U|=m,\\
\emptyset, &  \text{ otherwise }.
\end{cases}
\]
It is relevant here because, after removing the root of a rooted tree with fertility $m$,
one obtains an unordered collection of $m$ rooted subtrees.


\begin{definition}
For every integer $m\ge 0$, define the {\bf cycle-index polynomial} of the species $\mathrm{SET}_m$ by
\begin{equation*}
Z_{\mathrm{SET}_m}(p_1,\dots,p_m)
:=
\sum_{\lambda\vdash m}\frac{p_\lambda}{z_\lambda},
\end{equation*}
where, for a partition $\lambda\vdash m$ written in multiplicity notation as
$\lambda=1^{m_1}2^{m_2}\cdots$, the integer $m_r\in \mathbb{Z}_{\geq 0}$ denotes the number of parts of
$\lambda$ equal to $r$. Thus
$
m=\sum_{r\ge 1} r m_r,
$
and we set
\[
p_\lambda:=\prod_{r\ge 1}p_r^{m_r},
\qquad
z_\lambda:=\prod_{r\ge 1}r^{m_r}m_r!.
\]
\end{definition}

The cycle-index description of unordered sets gives the following functional equation
for the ordinary fibre generating series.

\begin{proposition}\label{prop:ordinary-functional}
The ordinary generating series $F(\uvec)$ satisfies the equivalent identities
\begin{align}
F(\uvec)
&=
\sum_{a\in A}\sum_{j\ge -1} u_{a,j}\,H_{j+1}(\uvec),\label{eq:ordinary-functional-H}\\
F(\uvec)
&=
\sum_{a\in A}\sum_{j\ge -1}
u_{a,j}\,[z^{j+1}]
\prod_{\bl\in\mathcal K_{-1}} (1-z\uvec^\bl)^{-F_\bl},\label{eq:ordinary-functional-coeff}\\
F(\uvec)
&=
\sum_{a\in A}\sum_{j\ge -1}
u_{a,j}\,
Z_{\mathrm{SET}_{j+1}}
\big(
F(\uvec^{[1]}),F(\uvec^{[2]}),\dots,F(\uvec^{[j+1]})
\big).\label{eq:ordinary-functional-cycle}
\end{align}
\end{proposition}

\begin{proof}
We first prove~\eqref{eq:ordinary-functional-H}.
Fix a multi-index $\kk$.
If $\wt(\kk)\neq -1$, then the coefficient of $\uvec^\kk$ in the right-hand side of
\eqref{eq:ordinary-functional-H} is $0$, because every monomial appearing in
$u_{a,j}H_{j+1}(\uvec)$ has weight
$
j+(j+1)(-1)=-1.
$
Hence both sides have zero coefficient at $\uvec^\kk$.
Assume now that $\wt(\kk)=-1$.
The coefficient of $\uvec^\kk$ in the right-hand side of~\eqref{eq:ordinary-functional-H} is
\begin{equation}\label{eq:coeff-H}
\sum_{a\in A}\sum_{\substack{j\ge -1\\ k_j^a\ge 1}}
[\uvec^{\kk-\ee_j^a}]\,H_{j+1}(\uvec).
\end{equation}
By Definition~\ref{def:Hm}, the coefficient in~\eqref{eq:coeff-H} is exactly
\begin{equation*}
\sum_{a\in A}\sum_{\substack{j\ge -1\\ k_j^a\ge 1}}
\;
\sum_{\substack{\nu \text{ finitely supported on }\mathcal K_{-1}\\
|\nu|=j+1,\ \Sigma(\nu)=\kk-\ee_j^a}}
\;
\prod_{\bl\in\mathcal K_{-1}}
\mathrm{Mlt}(F_\bl,\nu_\bl),
\end{equation*}
which is $F_\kk$ by Theorem~\ref{thm:ordinary-recursion}.
Thus the coefficient of $\uvec^\kk$ agrees on both sides of~\eqref{eq:ordinary-functional-H}.
This proves~\eqref{eq:ordinary-functional-H} coefficientwise.

The equivalence between~\eqref{eq:ordinary-functional-H} and~\eqref{eq:ordinary-functional-coeff}
is immediate from Proposition~\ref{prop:ordinary-product}, since
\begin{equation*}
H_{j+1}(\uvec)
=
[z^{j+1}]
\prod_{\bl\in\mathcal K_{-1}} (1-z\uvec^\bl)^{-F_\bl}.
\end{equation*}

It remains to prove~\eqref{eq:ordinary-functional-cycle}.
For this, recall the classical generating-series identity
\begin{equation}\label{eq:cycle-index-gf}
\sum_{m\ge 0} Z_{\mathrm{SET}_m}(p_1,\dots,p_m)\,z^m
=
\exp\left(\sum_{r\ge 1}\frac{p_r z^r}{r}\right).
\end{equation}
Let us justify~\eqref{eq:cycle-index-gf}.
Expanding the exponential gives
\begin{align}
\exp\left(\sum_{r\ge 1}\frac{p_r z^r}{r}\right)
&=
\prod_{r\ge 1}\exp\left(\frac{p_r z^r}{r}\right)
\notag\\
&=
\prod_{r\ge 1}\sum_{m_r\ge 0}\frac{p_r^{m_r} z^{r m_r}}{r^{m_r}m_r!}
\notag\\
&=
\sum_{(m_r)_{r\ge 1}}
\left(
\prod_{r\ge 1}\frac{p_r^{m_r}}{r^{m_r}m_r!}
\right)
z^{\sum_{r\ge 1} r m_r}.
\label{eq:cycle-index-expand}
\end{align}
Grouping the terms in~\eqref{eq:cycle-index-expand} according to the partition
\begin{equation*}
\lambda=1^{m_1}2^{m_2}\cdots
\end{equation*}
of $m=\sum_r r m_r$,  the coefficient of $z^m$ is precisely
\begin{equation*}
\sum_{\lambda\vdash m}\frac{p_\lambda}{z_\lambda}
=
Z_{\mathrm{SET}_m}(p_1,\dots,p_m).
\end{equation*}
This proves~\eqref{eq:cycle-index-gf}.

Now substitute
$p_r:=F(\uvec^{[r]})$ for $r\ge 1.$
By Proposition~\ref{prop:ordinary-product},~\eqref{eq:cycle-index-gf} becomes
\begin{equation*}
\sum_{m\ge 0} Z_{\mathrm{SET}_m}(p_1,\dots,p_m)\,z^m
=
\exp\left(\sum_{r\ge 1}\frac{z^r}{r}F(\uvec^{[r]})\right)
=
\sum_{m\ge 0} H_m(\uvec)\,z^m.
\end{equation*}
Comparing coefficients of $z^{j+1}$ gives
\begin{equation}\label{eq:H-as-cycle-index}
H_{j+1}(\uvec)
=
Z_{\mathrm{SET}_{j+1}}
\big(
F(\uvec^{[1]}),F(\uvec^{[2]}),\dots,F(\uvec^{[j+1]})
\big).
\end{equation}
Inserting~\eqref{eq:H-as-cycle-index} into~\eqref{eq:ordinary-functional-H} yields
\eqref{eq:ordinary-functional-cycle}.
\end{proof}

\section{Coefficient generating functions of the lowering derivation}\label{sec:coeff}
In the present section we introduce a different family of generating functions, namely the
coefficient generating functions attached to the lowering derivation $\bar\partial$ on the
full polynomial algebra $N(A)$.

We first recall the lowering derivation and the shift operation on multi-indices from~\cite{ZGM2026}; these are the
basic operations behind the coefficients studied in this section.

\begin{definition}\label{def:lowering}
Let $N(A)$ be the polynomial algebra generated by the variables $x_j^a$ with
$a\in A$ and $j\ge -1$.

\begin{enumerate}
\item Define a derivation $\bar\partial:N(A)\to N(A)$ by
\begin{equation}\label{eq:lowering-on-generators}
\bar\partial(x_j^a):=x_{j-1}^a
\quad\text{for }j\ge0,
\quad
\bar\partial(x_{-1}^a):=0.
\end{equation}

\item A {\bf lowering multi-index} is a multi-index $\boldsymbol\ell$ on $I_A$ satisfying
\[
\ell_{-1}^a=0
\quad
\text{for all }a\in A.
\]
For such an $\boldsymbol\ell$, its left shift is the multi-index
$\overleftarrow{\boldsymbol\ell}$ on $I_A$ defined by
\[
(\overleftarrow{\boldsymbol\ell})_j^a:=\ell_{j+1}^a
\quad
\text{for every }a\in A,\ j\ge -1.
\]
\end{enumerate}
\end{definition}

Unless explicitly stated otherwise, every occurrence of $\bl$ in this section denotes
a lowering multi-index.
%
Before expanding iterates of $\bar\partial$, we record when a target multi-index can be obtained
from a source multi-index by such a lowering shift.

\begin{proposition}\label{prop:unique-shift}
Let $\kk$ and $\bb$ be multi-indices on
$I_A$.
Assume that, for every decoration $a\in A$,
\begin{equation}\label{eq:decoration-balance}
\sum_{j\ge -1}k_j^a=\sum_{j\ge -1}b_j^a.
\end{equation}
Then there exists a lowering multi-index $\bl$ such that
\begin{equation}\label{eq:target-shift}
\bb=\kk-\bl+\overleftarrow{\bl}
\end{equation}
if and only if
\begin{equation}\label{eq:lambda-nonnegative}
\Lambda_j^a(\kk,\bb):=\sum_{m\ge j}(k_m^a-b_m^a)\ge 0
\quad
\text{for every } a\in A,\ j\ge 0.
\end{equation}
In that case, the  lowering multi-index $\bl = (\ell_j^a)_{a\in A, j\geq -1}$ is unique, and is given by
\begin{equation}\label{eq:l-explicit}
\ell_j^a=\Lambda_j^a(\kk,\bb)=\sum_{m\ge j}(k_m^a-b_m^a) \quad \text{for } j\geq 0.
\end{equation}
\end{proposition}

\begin{proof}
Assume first that~\eqref{eq:target-shift} holds for some finitely supported $\bl$.
For every $a\in A$ and every $j\ge 0$, the $j$-th coordinate of
\eqref{eq:target-shift} reads
\begin{equation*}
b_j^a=k_j^a-\ell_j^a+\ell_{j+1}^a.
\end{equation*}
Rearranging gives
\begin{equation}\label{eq:coordinate-difference}
k_j^a-b_j^a=\ell_j^a-\ell_{j+1}^a.
\end{equation}
Summing~\eqref{eq:coordinate-difference} over all $m\ge j$ and using the finite support of $\bl$,
we obtain
\begin{equation*}
\sum_{m\ge j}(k_m^a-b_m^a)
=
\sum_{m\ge j}(\ell_m^a-\ell_{m+1}^a)
=
\ell_j^a.
\end{equation*}
Thus~\eqref{eq:l-explicit} holds, and in particular
$\Lambda_j^a(\kk,\bb)=\ell_j^a\ge 0$.
This proves the necessity of~\eqref{eq:lambda-nonnegative}, as well as uniqueness.

Conversely, assume~\eqref{eq:decoration-balance} and~\eqref{eq:lambda-nonnegative}, and define
$\bl$ by~\eqref{eq:l-explicit}.
Since $\kk$ and $\bb$ are finitely supported, $\bl$ is finitely supported as well.
For every $a\in A$ and every $j\ge 0$, we have
\begin{align*}
\ell_j^a-\ell_{j+1}^a=\sum_{m\ge j}(k_m^a-b_m^a)-\sum_{m\ge j+1}(k_m^a-b_m^a)=k_j^a-b_j^a,
\end{align*}
and so
\begin{equation*}
b_j^a=k_j^a-\ell_j^a+\ell_{j+1}^a
\quad
\text{for every } j\ge 0.
\end{equation*}
It remains to check the coordinate $j=-1$.
By~\eqref{eq:decoration-balance} and~\eqref{eq:l-explicit} with $j=0$,
\begin{align}
b_{-1}^a
=
\sum_{m\ge -1}k_m^a-\sum_{j\ge 0}b_j^a
=
k_{-1}^a+\sum_{j\ge 0}(k_j^a-b_j^a)
=
k_{-1}^a+\ell_0^a,
\label{eq:minus-one-coordinate}
\end{align}
which is exactly the $(a,-1)$-coordinate of~\eqref{eq:target-shift}.
Therefore~\eqref{eq:target-shift} holds.
\end{proof}

We now expand the iterates of the lowering derivation in the monomial basis and define the
corresponding shift coefficients.  

\begin{proposition}\label{prop:C-recursion}
For every multi-index $\kk \in \NANMO$ and every integer
$r\ge 0$, there exist unique integers $C_{\kk,\bl}$ such that
\begin{equation}\label{eq:C-expansion}
\bar\partial^r \xx^\kk
=
\sum_{ |\bl|=r}
C_{\kk,\bl}\,
\xx^{\kk-\bl+\overleftarrow{\bl}}.
\end{equation}
These coefficients satisfy
\begin{equation}\label{eq:C-recursive}
C_{\kk,\mathbf 0}=1, \quad C_{\kk,\bl}
=
\sum_{a\in A}
\sum_{\substack{j\ge 0\\ \ell_j^a\ge 1}}
C_{\kk,\bl-\ee_j^a}
\big(k_j^a-\ell_j^a+1+\ell_{j+1}^a\big) 
\end{equation}
for every nonzero lowering multi-index $\bl$.
\end{proposition}
Here, in~\eqref{eq:C-expansion} and throughout, whenever an exponent multi-index has
a negative component, the corresponding monomial is understood to be zero.  Moreover,
we set
$
C_{\kk,\bl}:=0
$
whenever $\kk-\bl+\overleftarrow{\bl}$ has a negative component.

\begin{proof}[Proof of Proposition~\ref{prop:C-recursion}]
We first note that
\begin{equation}\label{eq:lowering-on-monomials}
\bar\partial(\xx^\mm)
=
\sum_{a\in A}\sum_{j\ge 0} m_j^a\,\xx^{\mm-\ee_j^a+\ee_{j-1}^a}
\end{equation}
for every multi-index $\mm$.
 
We prove~\eqref{eq:C-expansion} by induction on $r$.
If $r=0$, then
$
\bar\partial^0\xx^\kk=\xx^\kk,
$
which is exactly~\eqref{eq:C-expansion} with the unique term
$\bl=\mathbf 0$ and coefficient $C_{\kk,\mathbf 0}=1$.
This proves the first equation in~\eqref{eq:C-recursive}.
Assume now that~\eqref{eq:C-expansion} holds for some fixed $r\ge 0$.
Then
\begin{align}
\bar\partial^{r+1}\xx^\kk
&=
\bar\partial\!\left(
\sum_{  |\bl'|=r }
C_{\kk,\bl'}\,\xx^{\kk-\bl'+\overleftarrow{\bl'}}
\right) =
\sum_{  |\bl'|=r }
C_{\kk,\bl'}
\bar\partial\!\left(\xx^{\kk-\bl'+\overleftarrow{\bl'}}\right)
\notag\\
&=
\sum_{  |\bl'|=r}
C_{\kk,\bl'}
\sum_{a\in A}\sum_{j\ge 0}
\big(k_j^a-\ell_j'{}^a+\ell_{j+1}'{}^a\big)
\xx^{\kk-\bl'+\overleftarrow{\bl'}-\ee_j^a+\ee_{j-1}^a},
\label{eq:induction-step-1}
\end{align}
where the final equality follows from~\eqref{eq:lowering-on-monomials}.
Here, by the convention following~\eqref{eq:C-expansion}, terms with
$\kk-\bl'+\overleftarrow{\bl'}$ having a negative component have coefficient
$C_{\kk,\bl'}=0$ and do not contribute.  Hence, in applying
\eqref{eq:lowering-on-monomials}, we only need to consider those $\bl'$ for which
$\kk-\bl'+\overleftarrow{\bl'}$ is componentwise nonnegative.

Now fix $a\in A$ and $j\ge 0$, and set
$
\bl:=\bl'+\ee_j^a.
$
Then $|\bl|=r+1$ and
$
\overleftarrow{\bl}
=
\overleftarrow{\bl'}+\ee_{j-1}^a.
$
Hence
\begin{equation}\label{eq:monomial-reindex}
\xx^{\kk-\bl'+\overleftarrow{\bl'}-\ee_j^a+\ee_{j-1}^a}
=
\xx^{\kk-\bl+\overleftarrow{\bl}}.
\end{equation}
Moreover,
\begin{equation}\label{eq:coefficient-reindex}
k_j^a-\ell_j'{}^a+\ell_{j+1}'{}^a
=
k_j^a-\ell_j^a+1+\ell_{j+1}^a.
\end{equation}
Substituting~\eqref{eq:monomial-reindex} and~\eqref{eq:coefficient-reindex}
into~\eqref{eq:induction-step-1}, we obtain
\begin{equation*}
\bar\partial^{r+1}\xx^\kk
=
\sum_{  |\bl|=r+1}
\left(
\sum_{a\in A}\sum_{\substack{j\ge 0\\ \ell_j^a\ge 1}}
C_{\kk,\bl-\ee_j^a}
\big(k_j^a-\ell_j^a+1+\ell_{j+1}^a\big)
\right)
\xx^{\kk-\bl+\overleftarrow{\bl}}.
\end{equation*}
This proves both the existence of the coefficients $C_{\kk,\bl}$ and the recursion
\eqref{eq:C-recursive}.
The uniqueness of the family $(C_{\kk,\bl})_{\bl}$ follows from
Proposition~\ref{prop:unique-shift} since $\bl $ is unique in $\bb=\kk-\bl+\overleftarrow{\bl}$.
\end{proof}

For later use in the coproduct formula, it is convenient to renormalize the shift coefficients by
the factorial of the target multi-index.
For an integer-valued finitely supported family $\mathbf q$ on $I_A$, the factorial
$\mathbf q!$ is defined only when $\mathbf q\ge\mathbf0$.  Whenever a summand contains
$\mathbf q!$ or $1/\mathbf q!$, that summand is understood to occur only in this case;
otherwise the summand is omitted, or equivalently interpreted as zero.

\begin{definition}\label{def:D-coeff}
For every multi-index $\mathbf k$ and every lowering multi-index $\boldsymbol\ell$,
define
\begin{equation}\label{eq:D-definition}
D_{\mathbf k,\boldsymbol\ell}
:=
\begin{cases}
C_{\mathbf k,\boldsymbol\ell}
(\mathbf k-\boldsymbol\ell+\overleftarrow{\boldsymbol\ell})!,
&
\text{if }
\mathbf k-\boldsymbol\ell+\overleftarrow{\boldsymbol\ell}\ge\mathbf0,\\[1ex]
0,
&
\text{otherwise}.
\end{cases}
\end{equation}
\end{definition}


The renormalized coefficients satisfy a recursion obtained by rewriting the recursion for
$C_{\kk,\bl}$ in factorial-normalized form.

\begin{proposition}\label{prop:D-recursion}
The coefficients of Definition~\ref{def:D-coeff} satisfy
\begin{equation} \label{eq:D-recursive}
D_{\kk,\mathbf 0}=\kk!, \quad D_{\kk,\bl}
=
\sum_{a\in A}
\sum_{\substack{j\ge 0\\ \ell_j^a\ge 1}}
D_{\kk,\bl-\ee_j^a}
\big(k_{j-1}^a-\ell_{j-1}^a+\ell_j^a\big)
\end{equation}
for every nonzero lowering multi-index $\boldsymbol\ell$ such that
$
\mathbf k-\boldsymbol\ell+\overleftarrow{\boldsymbol\ell}\ge\mathbf0,
$
where
$D_{\mathbf k,\boldsymbol\ell-\mathbf e_j^a}$ is understood to be zero whenever
$
\mathbf k-(\boldsymbol\ell-\mathbf e_j^a)
+\overleftarrow{\boldsymbol\ell-\mathbf e_j^a}
$
has a negative component.
\end{proposition}


\begin{proof}
The first equation in~\eqref{eq:D-recursive} follows immediately from~\eqref{eq:C-recursive} and
\eqref{eq:D-definition}.
Let $\boldsymbol\ell$ be a nonzero lowering multi-index such that
$\mathbf k-\boldsymbol\ell+\overleftarrow{\boldsymbol\ell}\ge\mathbf0$.
By~\eqref{eq:C-recursive} and~\eqref{eq:D-definition},
\begin{align}
D_{\kk,\bl}
 =
C_{\kk,\bl}\,(\kk-\bl+\overleftarrow{\bl})!
=\sum_{a\in A}
\sum_{\substack{j\ge 0\\ \ell_j^a\ge 1}}
C_{\kk,\bl-\ee_j^a}
\big(k_j^a-\ell_j^a+1+\ell_{j+1}^a\big)
(\kk-\bl+\overleftarrow{\bl})!.
\label{eq:D-proof-1}
\end{align}
Fix one term in the right-hand side and set
\begin{equation*}
\alpha:=\kk-\bl+\overleftarrow{\bl},
\quad
\beta:=\kk-(\bl-\ee_j^a)+\overleftarrow{\bl-\ee_j^a}.
\end{equation*}
Since $j\ge 0$, we have
\begin{equation*}
\beta=\alpha+\ee_j^a-\ee_{j-1}^a.
\end{equation*}
Therefore all coordinates of $\alpha$ and $\beta$ coincide except the
$(a,j)$- and $(a,j-1)$-coordinates, and more precisely
\begin{equation*}
\beta_j^a=\alpha_j^a+1,
\quad
\beta_{j-1}^a=\alpha_{j-1}^a-1.
\end{equation*}
If $\beta$ has a negative component, then necessarily
$
\alpha_{j-1}^a=0.
$
In this case
\[
k_{j-1}^a-\ell_{j-1}^a+\ell_j^a=\alpha_{j-1}^a=0,
\]
and by Definition~\ref{def:D-coeff} the term
$D_{\kk,\bl-\ee_j^a}$ is zero.  Hence the corresponding summand in the desired
recursion is zero.  Moreover, by the convention after~\eqref{eq:C-expansion},
$C_{\kk,\bl-\ee_j^a}=0$, since the target
$$\kk-(\bl-\ee_j^a)+\overleftarrow{\bl-\ee_j^a}=\beta$$
has a negative component.  Thus the corresponding summand in~\eqref{eq:D-proof-1}
also vanishes.

It remains to consider the case where $\beta$ is componentwise nonnegative.  Then
$\alpha_{j-1}^a\ge1$, and all factorials below are well-defined.  Since
\[
\beta_j^a=\alpha_j^a+1,
\qquad
\beta_{j-1}^a=\alpha_{j-1}^a-1,
\]
we have
\[
\beta!
=
\alpha!\frac{\alpha_j^a+1}{\alpha_{j-1}^a}.
\]
Using
\[
\alpha_j^a+1
=
k_j^a-\ell_j^a+1+\ell_{j+1}^a, \quad \alpha_{j-1}^a
=
k_{j-1}^a-\ell_{j-1}^a+\ell_j^a,
\]
we obtain
\begin{equation}\label{eq:key-factorial-identity}
\bigl(k_j^a-\ell_j^a+1+\ell_{j+1}^a\bigr)\alpha!
=
\bigl(k_{j-1}^a-\ell_{j-1}^a+\ell_j^a\bigr)\beta!.
\end{equation} 
Substituting~\eqref{eq:key-factorial-identity} into~\eqref{eq:D-proof-1} gives
\begin{align*}
D_{\kk,\bl}
=
\sum_{a\in A}
\sum_{\substack{j\ge 0\\ \ell_j^a\ge 1}}
C_{\kk,\bl-\ee_j^a}\,
\beta!\,
\big(k_{j-1}^a-\ell_{j-1}^a+\ell_j^a\big)
=
\sum_{a\in A}
\sum_{\substack{j\ge 0\\ \ell_j^a\ge 1}}
D_{\kk,\bl-\ee_j^a}
\big(k_{j-1}^a-\ell_{j-1}^a+\ell_j^a\big),
\end{align*}
which is exactly~\eqref{eq:D-recursive}.
\end{proof}

By~\cite[Theorem~13]{ZGM2026}, for every multi-index $\kk$ with
$\wt(\kk)=-1$, one has
\begin{equation}\label{eq:lot-coproduct-explicit}
\Delta_{\mathrm{LOT}}(\xx^\kk)
=
\sum_{r\ge0}
\sum_{\substack{
\kk=\kk^{1}+\cdots+\kk^{r}+\bb\\
\wt(\kk^{i})=-1,\ \forall i\in [r]
}}
\frac{\kk!}{\sigma(\xx^{\kk^{1}}\ext\cdots\ext\xx^{\kk^{r}})\,\bb!}
\xx^{\kk^{1}}\ext\cdots\ext\xx^{\kk^{r}} \otimes \bar\partial^{\,r}\xx^\bb,
\end{equation}
where $\sigma(\xx^{\kk^{1}}\ext\cdots\ext\xx^{\kk^{r}})$ is the symmetry factor of $\xx^{\kk^{1}}\ext\cdots\ext\xx^{\kk^{r}}$.
As a direct consequence of Proposition~\ref{prop:C-recursion}, we obtain a refinement of the above formula. 

\begin{corollary}\label{coro:cut-coproduct-refined}
For every multi-index $\kk$ with $\wt(\kk)=-1$, the coproduct $\Delta_{\mathrm{LOT}}$ given in~(\ref{eq:lot-coproduct-explicit})
admits the refinements  
\begin{align}
\Delta_{\mathrm{LOT}}(\xx^\kk)
&=
\sum_{r\ge 0}
\sum_{\substack{
\kk=\kk^{1}+\cdots+\kk^{r}+\bb\\
\wt(\kk^{i})=-1,\ \forall i\in [r] 
}}
\frac{\kk!}{\sigma(\xx^{\kk^{1}}\ext\cdots\ext \xx^{\kk^{r}})\,\bb!}
\sum_{\substack{\bl\ge \mathbf0\\ |\bl|=r}}
C_{\bb,\bl}\,
\xx^{\kk^{1}}\ext\cdots\ext \xx^{\kk^{r}}\otimes \xx^{\bb-\bl+\overleftarrow{\bl}},
\label{eq:cut-coproduct-C}
\\[1ex]
&=
\sum_{r\ge 0}
\sum_{\substack{
 \kk=\kk^{1}+\cdots+\kk^{r}+\bb\\
\wt(\kk^{i})=-1,\ \forall i\in [r]
}}
\frac{\kk!}{\sigma(\xx^{\kk^{1}}\ext\cdots\ext \xx^{\kk^{r}})\,\bb!}
\sum_{\substack{\bl\ge\mathbf0\\ |\bl|=r\\
\bb-\bl+\overleftarrow{\bl}\ge\mathbf0}}
\frac{D_{\bb,\bl}}{(\bb-\bl+\overleftarrow{\bl})!}\,
\xx^{\kk^{1}}\ext\cdots\ext \xx^{\kk^{r}}\otimes \xx^{\bb-\bl+\overleftarrow{\bl}}.
\label{eq:cut-coproduct-D}
\end{align}
\end{corollary}

\begin{proof}
By Proposition~\ref{prop:C-recursion}, for every remainder $\bb$ and every integer $r\ge 0$,
\begin{equation}\label{eq:barpartial-bb-expand}
\bar\partial^{\,r}\xx^\bb
=
\sum_{\substack{\bl\ge \mathbf0, |\bl|=r}}
C_{\bb,\bl}\,
\xx^{\bb-\bl+\overleftarrow{\bl}}.
\end{equation}
Substituting~\eqref{eq:barpartial-bb-expand} into~\eqref{eq:lot-coproduct-explicit} gives
\eqref{eq:cut-coproduct-C}.
For those $\bl$ such that
$\bb-\bl+\overleftarrow{\bl}\ge\mathbf0$, Definition~\ref{def:D-coeff} gives
\begin{equation}\label{eq:D-over-C}
D_{\bb,\bl}
=
C_{\bb,\bl}\,
(\bb-\bl+\overleftarrow{\bl})!.
\end{equation}
Whenever $C_{\bb,\bl}\neq 0$, the multi-index $\bb-\bl+\overleftarrow{\bl}$ is precisely the exponent
multi-index of a monomial appearing in~\eqref{eq:barpartial-bb-expand}, hence all its coordinates are
nonnegative and its factorial is well-defined.
Solving~\eqref{eq:D-over-C} for $C_{\bb,\bl}$ and inserting the result into~\eqref{eq:cut-coproduct-C}
yields~\eqref{eq:cut-coproduct-D}.
\end{proof}

We next package the iterates of $\bar\partial$ into coefficient generating functions.

\begin{definition}\label{def:coefficient-gf}
For every multi-index $\kk$, define the {\bf polynomial in $u$ with coefficients in $N(A)$} by
\begin{equation}\label{eq:Cgf-definition}
\mathcal C_{\kk}(u;\xx)
:=
\sum_{r\ge 0}\frac{u^r}{r!}\bar\partial^r \xx^\kk.
\end{equation}
Since $\bar\partial$ is locally nilpotent on each monomial, the sum in
\eqref{eq:Cgf-definition} is finite.
For every target multi-index $\bb$, define the scalar coefficient generating functions
\begin{equation*}
\mathcal C_{\kk,\bb}(u):=[\xx^\bb]\mathcal C_{\kk}(u;\xx),
\quad
\mathcal D_{\kk,\bb}(u):=\bb!\,\mathcal C_{\kk,\bb}(u).
\end{equation*}
\end{definition}

The coefficient generating function admits a simple factorized expression because the exponential
of a derivation acts multiplicatively.

\begin{proposition}
For every multi-index $\kk$,
\begin{equation}\label{eq:factorized-Cgf}
\mathcal C_{\kk}(u;\xx)
=
\prod_{a\in A}\prod_{j\ge -1}
\left(
\sum_{m=0}^{j+1}\frac{u^m}{m!}x_{j-m}^a
\right)^{k_j^a}.
\end{equation}
Equivalently,
\begin{equation}\label{eq:factorized-Cgf-s}
\mathcal C_{\kk}(u;\xx)
=
\prod_{a\in A}\prod_{j\ge -1}
\left(
\sum_{s=-1}^{j}\frac{u^{j-s}}{(j-s)!}x_s^a
\right)^{k_j^a}.
\end{equation}
\end{proposition}

\begin{proof}
Let $P,Q\in N(A)$.
Since $\bar\partial$ is a derivation, the iterated Leibniz rule holds:
\begin{equation*}
\bar\partial^r(PQ)
=
\sum_{s=0}^{r}\binom{r}{s}\bar\partial^s(P)\bar\partial^{r-s}(Q).
\end{equation*}
Therefore
\begin{align*}
\sum_{r\ge 0}\frac{u^r}{r!}\bar\partial^r(PQ)
&=
\sum_{r\ge 0}\frac{u^r}{r!}
\sum_{s=0}^{r}\binom{r}{s}\bar\partial^s(P)\bar\partial^{r-s}(Q)
\notag\\
&=
\sum_{r\ge 0}\sum_{s=0}^{r}
\frac{u^s}{s!}\bar\partial^s(P)\,
\frac{u^{r-s}}{(r-s)!}\bar\partial^{r-s}(Q)
\notag\\
&=
\left(\sum_{s\ge 0}\frac{u^s}{s!}\bar\partial^s(P)\right)
\left(\sum_{t\ge 0}\frac{u^t}{t!}\bar\partial^t(Q)\right).
\end{align*}
Hence the map
\begin{equation*}
E_u:=\sum_{r\ge 0}\frac{u^r}{r!}\bar\partial^r: N(A) \longrightarrow N(A)
\end{equation*}
is an algebra morphism.
Now let $a\in A$ and $j\ge -1$.
By~\eqref{eq:lowering-on-generators},
\begin{equation*}
\bar\partial^m(x_j^a)=
\begin{cases}
x_{j-m}^a, & 0\le m\le j+1,\\[1ex]
0, & m\ge j+2.
\end{cases}
\end{equation*}
Therefore
\begin{equation*}
E_u(x_j^a)
=
\sum_{m=0}^{j+1}\frac{u^m}{m!}x_{j-m}^a.
\end{equation*}
Since
\begin{equation*}
\xx^\kk=\prod_{a\in A}\prod_{j\ge -1}(x_j^a)^{k_j^a},
\end{equation*}
and the product is finite, the multiplicativity of $E_u$ gives
\begin{align*}
\mathcal C_{\kk}(u;\xx)
=
E_u(\xx^\kk)
=
\prod_{a\in A}\prod_{j\ge -1}E_u(x_j^a)^{k_j^a}
=
\prod_{a\in A}\prod_{j\ge -1}
\left(
\sum_{m=0}^{j+1}\frac{u^m}{m!}x_{j-m}^a
\right)^{k_j^a},
\end{align*}
which is exactly~\eqref{eq:factorized-Cgf}.
The form~\eqref{eq:factorized-Cgf-s} is obtained by the change of variable
$s=j-m$.
\end{proof}

To extract individual transition coefficients from the factorized expression, we introduce
transport arrays from a source multi-index to a target multi-index.

\begin{definition}
Let $\kk$ and $\bb$ be multi-indices on
$I_A$.
A {\bf transport array} from $\kk$ to $\bb$ is a family
\begin{equation*}
\mathbf n=(n_{a,j,s})_{a\in A,\,-1\le s\le j}
\end{equation*}
of nonnegative integers such that
\begin{equation*}
\sum_{s=-1}^{j} n_{a,j,s}=k_j^a
\quad
\text{for every } a\in A,\ j\ge -1,
\end{equation*}
and
\begin{equation}\label{eq:transport-target}
\sum_{j\ge s} n_{a,j,s}=b_s^a
\quad
\text{for every } a\in A,\ s\ge -1.
\end{equation}
The set of all transport arrays from $\kk$ to $\bb$ is denoted by
$\mathcal N(\kk,\bb).$
\end{definition}

The scalar coefficient generating functions can be identified with the shift coefficients as follows.

\begin{theorem}\label{thm:explicit-transition}
For any multi-indices $\kk$ and $\bb$, one has
\begin{equation}\label{eq:explicit-transition}
\mathcal C_{\kk,\bb}(u)
=
\sum_{\mathbf n\in\mathcal N(\kk,\bb)}
\left(
\prod_{a\in A}\prod_{j\ge -1}
\frac{k_j^a!}{\prod_{s=-1}^{j} n_{a,j,s}!}
\right)
\frac{
u^{\sum\limits_{a\in A}\sum\limits_{j\ge -1}\sum\limits_{s=-1}^{j}(j-s)n_{a,j,s}}
}{
\prod\limits_{a\in A}\prod\limits_{j\ge -1}\prod\limits_{s=-1}^{j}(j-s)!^{\,n_{a,j,s}}
}.
\end{equation}
In particular, $\mathcal C_{\kk,\bb}(u)=0$ if $\mathcal N(\kk,\bb)=\emptyset$.
\end{theorem}

\begin{proof}
Starting from~\eqref{eq:factorized-Cgf-s}, we expand each factor by the multinomial theorem:
\begin{align}
\left(
\sum_{s=-1}^{j}\frac{u^{j-s}}{(j-s)!}x_s^a
\right)^{k_j^a}
&=
\sum_{\substack{(n_{a,j,s})_{-1\le s\le j}\\ \sum_{s=-1}^{j}n_{a,j,s}=k_j^a}}
\frac{k_j^a!}{\prod_{s=-1}^{j} n_{a,j,s}!}
\prod_{s=-1}^{j}
\left(
\frac{u^{j-s}}{(j-s)!}x_s^a
\right)^{n_{a,j,s}}.
\label{eq:multinomial-step}
\end{align}
Since $\kk$ has finite support, only finitely many pairs $(a,j)$ contribute nontrivially.        Multiplying
\eqref{eq:multinomial-step} over all $(a,j)$ gives
\begin{align}
\mathcal C_{\kk}(u;\xx)
&=
\sum_{\substack{\mathbf n=(n_{a,j,s})\\ \sum_{s=-1}^{j}n_{a,j,s}=k_j^a}}
\left(
\prod_{a\in A}\prod_{j\ge -1}
\frac{k_j^a!}{\prod_{s=-1}^{j} n_{a,j,s}!}
\right)
\frac{
u^{\sum\limits_{a\in A}\sum\limits_{j\ge -1}\sum\limits_{s=-1}^{j}(j-s)n_{a,j,s}}
}{
\prod\limits_{a\in A}\prod\limits_{j\ge -1}\prod\limits_{s=-1}^{j}(j-s)!^{\,n_{a,j,s}}
}
\prod_{a\in A}\prod_{s\ge -1}
(x_s^a)^{\sum_{j\ge s}n_{a,j,s}}.
\label{eq:Cgf-expanded}
\end{align}
The exponent of $x_s^a$ in~\eqref{eq:Cgf-expanded} is exactly
$\sum_{j\ge s}n_{a,j,s}$.
Therefore the coefficient of the monomial $\xx^\bb$ is obtained precisely by imposing the target
conditions~\eqref{eq:transport-target}.
This yields~\eqref{eq:explicit-transition}.
The final assertion is immediate.
\end{proof}

\begin{corollary}
Let $\kk$ and $\bb$ be multi-indices on $I_A$.
\begin{enumerate}
\item If there exists no lowering multi-index $\bl$  such that
\begin{equation}\label{eq:reachable-target}
\bb=\kk-\bl+\overleftarrow{\bl},
\end{equation}
then
\begin{equation}\label{eq:transition-zero}
\mathcal C_{\kk,\bb}(u)=0,
\quad
\mathcal D_{\kk,\bb}(u)=0.
\end{equation} 

\item If such a lowering multi-index $\bl$ exists, then it is unique by Proposition~\ref{prop:unique-shift}, and
\begin{equation}\label{eq:transition-C}
\mathcal C_{\kk,\bb}(u)
=
\frac{u^{|\bl|}}{|\bl|!}\,C_{\kk,\bl}, \quad \mathcal D_{\kk,\bb}(u)
=
\frac{u^{|\bl|}}{|\bl|!}\,D_{\kk,\bl}.
\end{equation}
\end{enumerate}
\end{corollary}

\begin{proof}
By Definition~\ref{def:coefficient-gf} and Proposition~\ref{prop:C-recursion},
\begin{align*}
\mathcal C_{\kk}(u;\xx) =\sum_{r\ge 0}\frac{u^r}{r!}\bar\partial^r \xx^\kk
 =\sum_{r\ge 0}\frac{u^r}{r!}
\sum_{\substack{|\bl|=r}}
C_{\kk,\bl}\,\xx^{\kk-\bl+\overleftarrow{\bl}}
=\sum_{\bl}
\frac{u^{|\bl|}}{|\bl|!}
C_{\kk,\bl}\,\xx^{\kk-\bl+\overleftarrow{\bl}}.
\end{align*}
Taking the coefficient of $\xx^\bb$ and using Proposition~\ref{prop:unique-shift}, we obtain
\eqref{eq:transition-zero} and the first identity in~\eqref{eq:transition-C}.
Finally, if~\eqref{eq:reachable-target} holds, then
\begin{equation*}
\bb!=(\kk-\bl+\overleftarrow{\bl})!,
\end{equation*}
and so by Definitions~\ref{def:D-coeff} and~\ref{def:coefficient-gf},
\begin{align*}
\mathcal D_{\kk,\bb}(u)
=\bb!\,\mathcal C_{\kk,\bb}(u)
=(\kk-\bl+\overleftarrow{\bl})!\,
\frac{u^{|\bl|}}{|\bl|!}\,C_{\kk,\bl}
=\frac{u^{|\bl|}}{|\bl|!}\,D_{\kk,\bl},
\end{align*}
which is the second identity in~\eqref{eq:transition-C}.
\end{proof}

\noindent
{\bf Acknowledgments.} This work is supported by the National Natural Science Foundation of China (12571019), the Natural Science Foundation of Gansu Province (25JRRA644) and Innovative Fundamental Research Group Project of Gansu Province (23JRRA684).

\vskip 0.2in

\noindent
{\bf Declaration of interests. } The authors have no conflicts of interest to disclose.

\noindent
{\bf Data availability. } Data sharing is not applicable as no new data were created or analyzed.


\end{CJK*}
\end{document}